\documentclass[a4paper,11pt]{amsart} 
\usepackage[margin=25mm]{geometry}
\usepackage[english,activeacute]{babel} 
\usepackage{color}
\usepackage{graphics}
\usepackage{graphicx}
\usepackage{hyperref}
\usepackage{longtable}
\usepackage{amsmath}
\usepackage{epigraph}
\usepackage{indentfirst}
\usepackage{verbatim}
\usepackage{array}
\usepackage{amssymb}
\usepackage{subfigure}
\usepackage{fancyhdr}
\usepackage{enumitem}
\DeclareMathAlphabet{\mathpzc}{OT1}{pzc}{m}{it} 
\usepackage{mathrsfs}
\usepackage[small,bf]{caption}
\usepackage{cancel}
\usepackage{braket} 
\usepackage{mathrsfs}  
\usepackage{epstopdf}
\usepackage{amsthm}
\usepackage{dsfont}
\usepackage{mathtools} 
\usepackage{cancel}
\allowdisplaybreaks

\newtheorem{theorem}{Theorem}[section]
\newtheorem{corollary}[theorem]{Corollary}
\newtheorem{lemma}[theorem]{Lemma}
\newtheorem{proposition}[theorem]{Proposition}
\theoremstyle{definition}

\newtheorem{remark}{Remark}

\newcommand{\R}{\mathbb{R}}
\newcommand{\C}{\mathbb{C}}

\newcommand{\T}{\mathbb{T}}
\newcommand{\walk}{\mathsf{s}}
\newcommand{\walks}{\mathsf{S}}

\newcommand{\pf}{\mathsf{Z}}
\newcommand{\mart}{\mathsf{M}}
\newcommand{\martx}{\mathsf{X}}

\newcommand{\eterm}{\mathsf{E}}
\newcommand{\free}{\mathsf{f}}
\newcommand{\amin}{\alpha_{\textsc{min}}}
\newcommand{\fone}{f_{\textsc{I}}}
\newcommand{\ftwo}{f_{\textsc{II}}}
\newcommand{\fthree}{f_{\textsc{III}}}
\newcommand{\bigfree}{\textsf{F}}
\newcommand{\esp}{\mathbb{E}}
\newcommand{\p}{\mathbb{P}}

\newcommand{\calf}{\mathcal{F}}

\newcommand{\mom}{\mathsf{m}}
\newcommand{\bmom}{\widetilde{\mom}}
\newcommand{\xibg}{\xi_{\beta,\gamma}}
\newcommand{\lrad}{\lambda_{\R}}
\newcommand{\lph}{\lambda_{\C}}

\title[Directed Polymers in Complex-Valued Environments on Trees]{Directed Polymers in Complex-Valued Random Environments on the Tree}
\author[Leonardo Medina-Espinosa and Gregorio R. Moreno Flores]{Leonardo Medina-Espinosa$^1$ and Gregorio R. Moreno Flores$^{2}$}

\address{Facultad de Matem\'aticas\\
Pontificia Universidad Cat\'olica de Chile\\
Vicu\~na Mackenna 4860, Macul\\
Santiago, Chile}

\email{lpmedina@mat.uc.cl, grmoreno@mat.uc.cl}

\thanks{$^{1,2}$ Facultad de Matem\'aticas, Pontificia Universidad Cat\'olica de Chile.}

\thanks{$^{1,2}$  Partially supported by Fondecyt grant 1211189}

\begin{document}

\begin{abstract}
	We consider a model of directed polymers on regular trees with complex-valued random weights introduced by Cook and Derrida \cite{CD-Complex} and studied mathematically by Derrida, Evans and Speer \cite{DES-Complex}. In addition to the usual weak-disorder and strong-disorder regimes, the phase diagram of the model contains a third region due to the effects of the random phases that cannot be observed in the model with positive weights.
	
	In this work, we extend the results of \cite{DES-Complex} in two directions. First, we remove the hypothesis on the independence of the random radii and random phases of the environment from most of the phase diagram, and second, under mild assumptions on the law of the environment, we strengthen the convergence of the free energy (shown to hold in probability in \cite{DES-Complex}) to an almost sure convergence. Along the way, many of the arguments of \cite{DES-Complex} are simplified.
\end{abstract}

\maketitle

\setcounter{tocdepth}{1}
\tableofcontents


\section{Model and main results}

Let $\left( \T, \mathcal{E} \right)$ be a $b$-ary tree, where \( \T \) denotes the set of nodes or sites, and \(b>1\). For $n\geq 0$, let $\T_n$ denote the set of all sites in generation $n$, i.e. those that are at graph distance $n$ from the root. Let $\walks_n$ be the set of nearest-neighbour directed paths of length $n$ i.e. if $\walk\in\walks_n$, then, $\walk = \left( \walk_0,\walk_1,\cdots,\walk_n \right)$ with $\walk_j \in \T_j$, $j=0,\cdots,n$ and $\displaystyle \left| \walk_j-\walk_{j-1} \right|=1$ for each $j=1,\cdots, n$.

Let $\{\xi(x): \ x\in\T\}$ be a family of complex-valued i.i.d. random variables defined on some probability space $(\Omega,\calf,\p)$ such that
\begin{eqnarray}
\label{eq:finite-moments}
	\esp
	\left[ \left| \xi \right|^{\alpha} \right]
	< \infty,
\end{eqnarray}
for all $\alpha>0$, where $\esp$ denotes the expected value with respect to $\p$. With a slight abuse of notation, we denote $\xi=\{\xi(x): \ x\in\T\}$. To be more precise, we may take $\Omega=\C^{\T}$, $\calf$ as the Borel $\sigma$-algebra on $\Omega$, $\p = \textsf{P}_0^{\otimes \T}$ for some probability measure $\textsf{P}_0$ defined the Borel $\sigma$-algebra of $\C$ and take $\xi$ as the canonical process. In some cases, it will be convenient to write $\xi(x)=e^{\omega(x)+\iota \theta(x)}$ and denote  $\omega=\{\omega(x):\, x\in\T\}$ and $\theta=\{\theta(x):\, x\in\T\}$.

For each realization of the environment $\xi$, we define the partition function $\pf_n(\xi)$ as
\begin{eqnarray}
\label{eq:partition-function}
	\pf_n(\xi)
	=
	\sum_{\walk\in\walks_n}
	\prod^n_{t=1} \xi(\walk_t).
\end{eqnarray}
Cook and Derrida introduced this model in \cite{CD-Complex} in relation with the Lyapunov exponents of large sparse random matrices. It was later solved on mathematical grounds by Derrida, Evans and Speer \cite{DES-Complex}. It can be seen as a generalization of directed polymers with positive random weights, introduced in the eighties as a model of phase separation lines in materials with impurities \cite{HH}. In the words of the authors of \cite{DES-Complex}: ``This generalization seems reasonable as a model for the hopping conductivity of strongly localized electrons, since the transmission of such electrons is dominated by directed paths,, and interference effects are produced when the contributions of the individual paths are added." In fact, except for a normalization, the partition function \eqref{eq:partition-function} can be understood as a path integral for the propagator of a mean-field discrete unitary Anderson model in a space-time random environment.

Our comprehension of directed polymers with positive weights (on the lattice or on the three) has made considerable progresses since their introduction. Their mathematical study can traced back to \cite{IS}, and a complete account of the principal works on the subject up to 2016 can be found in the book \cite{C-SF}. In this case, it is common to write the environment as $\xi=e^{\beta \omega}$, where $\beta>0$ represents the inverse temperature of the system. The partition function then corresponds to the normalizing constant of a random measure on paths known as the polymer measure. It is known that there is a phase transition: for small values of $\beta$, the polymer paths are diffusive while, for large values of $\beta$, the paths localize near energetically favorable regions (see \cite{Ch} for a recent work on localization). These two antagonistic phases are known as the weak-disorder and strong-disorder regimes. The critical value of $\beta$ is unknown in general, except for dimensions one and two, where it is equal to zero \cite{CV, L}, and on the tree, where it can be given explicitly as a function of the law of the environment and the branching number, along with the precise value of the free energy \cite{KP,BPP}. We recall this fact in Section \ref{sec:comparison}.

For complex weights, the authors of \cite{CD-Complex} conjectured the existence of an additional regime dominated by interferences coming from the random phases, which cannot be observed in the model with positive weights. This was later confirmed in \cite{DES-Complex}, where the free energy was computed explicitly, and the three regions were identified. The authors observe that this phase diagram does not coincide with other predictions from the physics literature, for instance, \cite{GB}, where, based on replica methods, a phase diagram consisting of five different regions was obtained, although not in the mean-field approximation. This raises the question of whether the model on the lattice could display a richer phase diagram, unlike the model with positive weights, where the lattice and tree cases share the same phase transition. We refer the reader to \cite{CD-Complex, DES-Complex} and references therein for a deeper review of the physics literature.

Much like the usual directed polymers on the tree, we note that the model can be interpreted as a multiplicative cascade inducing a (complex-valued) measure on the interval \cite{Barral1,Barral2,Barral3}. Furthermore, a phase diagram with three similar regions can be found in related models such as complex multiplicative chaos \cite{LRV-Complex, Lacoin1, Lacoin2} (see also \cite{HK-BBM} and references therein).

\vspace{1ex}

In this work, we partially extend the results of \cite{DES-Complex} in two different directions. 
First, the authors of \cite{DES-Complex} assume that the random phases and radii are independent, i.e., $\xi(x)=e^{\omega(x)+\iota \theta(x)}$, where the families of random variables $\{\omega(x):\, x\in\mathbb{T}\}$ and $\{\theta(x):\, x\in\mathbb{T}\}$ are independent. We remove this restriction from most of the phase diagram, except from a piece of the strong disorder region where the model with weights $|\xi(x)|$ is in the weak disorder regime but where sufficiently disordered random phases induce a phase transition. 
Second, under mild assumptions on the law of the environment, we show that the convergence to the free energy, which was shown to hold in probability in \cite{DES-Complex}, can be strengthened to almost sure convergence. 
We also show that the standard martingale techniques used for positive weights can still be applied in the weak disorder region, yielding a straightforward computation of the free energy and almost sure convergence without additional hypotheses beyond those of \cite{DES-Complex}. In order to keep the article self-contained, several proofs from \cite{DES-Complex} are reproduced, often with significant simplifications. Proposition \ref{thm:joint-convergence} below clarifies the main elements of our version of the scheme of proof used in \cite{DES-Complex}.

\vspace{1ex}

Finally, we recall that the work \cite{MO} characterizes the localization properties of directed polymers with positive weights on the tree. There, it is shown that, in the weak disorder region, the partition function is supported on a sub-tree whose growth rate is computed explicitly, decreases with $\beta$, and reduces to a single path in the strong disorder region. Obtaining analogous results in our case is a very natural question. In particular, we would expect that the random phases strengthen the localization in the weak disorder regime, up to the boundary of the region. We should also observe localization effects caused by interferences in the new region of the phase diagram. Recalling the relations with the Anderson model, it would be interesting to determine if this behavior shares similarities with Anderson's localization. We defer these questions to future work.

We now formulate the main results of this work. We first state the convergence to the free energy in a general setting, then specifically for the case of independent phases and radii, where a clearer picture of the phase diagram can be given. Then, we recall the results for the model with positive weights and relate them to the polymer with weights $|\xi(x)|$.


\subsection{The free energy}

In the rest of this work, the hypothesis \eqref{eq:finite-moments} is always assumed to hold.
Let
\begin{eqnarray*}
	G(\alpha) = \frac{\ln b + \ln \esp \left[ \left| \xi \right|^{\alpha} \right]}{\alpha},
	\qquad
	\alpha \geq 0.
\end{eqnarray*}
Let $\amin$ be the unique minimizer of $G$ if it exists and $\amin=\infty$ otherwise. It is easy to show that $G$ is strictly decreasing in $(0,\amin]$ and strictly increasing on $[\amin,\infty)$ (provided the later interval is meaningful).

Our hypothesis on the environment is slightly more general than the one from \cite{DES-Complex}. We say that $\mathsf{P}_0$ is \textit{uniformly continuous} if, for all $\varepsilon>0$, there exists $\delta>0$ such that
\begin{eqnarray*}
	\mathsf{P}_0
	\left[ \xi \in B(z;\delta) \right] \leq  \varepsilon,
\end{eqnarray*}
for all $z\in \mathbb{C}$, where $B(z;\delta) \subset \mathbb{C}$ denotes the ball with radius $\delta$ centered in $z$. 

The following is our main result, which extends Theorems 6.5 and 7.4 from \cite{DES-Complex}.
\begin{theorem}
\label{thm:main}
	Assume that the law $\mathsf{P}_0$ is uniformly continuous with respect to the Lebesgue measure. Then, the free energy
	\begin{eqnarray}
	\label{eq:free-energy}
		\free(\xi)
		=
		\lim_{n\to\infty}
		\frac{1}{n}\ln |\pf_n(\xi)|,
	\end{eqnarray}
	exists in probability and assumes the following values:
	\vspace{1ex}
	\begin{enumerate}
		\item (\textsc{Region} $\mathcal{R}1$) If there exists $\alpha \in (1,2]$ such that $G(\alpha) < \ln \left( b \left| \esp \left[ \xi \right] \right| \right)$,
			\begin{eqnarray*}
				\free(\xi) = \fone := \ln \left( b \left| \esp \left[ \xi \right] \right| \right).
			\end{eqnarray*}
		
		\vspace{1ex}
		
		\item (\textsc{Region} $\mathcal{R}2$) 
		If $\amin<1$, or, if $1 \leq \amin < 2$, $G \left( \amin \right)>\ln \left( b \left| \esp \left[ \xi \right] \right| \right)$ and the families of random variables $\omega$ and $\theta$ are independent,
			\begin{eqnarray*}
				\free(\xi) = \ftwo := G \left( \amin \right).
			\end{eqnarray*}
		
		\vspace{1ex}
		
		\item (\textsc{Region} $\mathcal{R}3$) If $\amin>2$ and $G(2)=\frac{1}{2}\ln\left( b \esp \left[ \left| \xi \right|^2 \right]\right)>\ln \left( b \left| \esp \left[ \xi \right] \right| \right)$,
			\begin{eqnarray*}
				\free(\xi)
				=
				\fthree
				:=
				\ln \left( b \esp \left[ \left| \xi \right|^2 \right] \right).
			\end{eqnarray*}
	\end{enumerate}
	Furthermore, the limit \eqref{eq:free-energy} holds $\p$-almost surely in the region $\mathcal{R}1$ and in the part of the region $\mathcal{R}2$ where $\amin<1$. 
\end{theorem}
The above result was obtained in \cite{DES-Complex} under the hypothesis that the random variables $\omega$ and $\theta$ are independent. As can be seen above, we remove this restriction from most of the phase diagram, except from the part of the region $\mathcal{R}2$ where $1\leq \amin <2$. As we will observe later, under this hypothesis, the model with positive weights $|\xi(x)|$ is in the weak-disorder regime. Here, the addition of sufficiently disordered random phases can induce strong disorder. It is then quite natural to expect that some hypotheses are required to observe such a behavior.
\begin{remark}
	Note that, if the law of the environment is non-trivial,
	\begin{eqnarray*}
		G(1) 
		= 
		\ln\left( b \esp[|\xi|] \right)
		>
		\ln\left( b|\esp[\xi]| \right),
	\end{eqnarray*}
	so that, if $\amin\leq 1$, $G(\alpha)>G(1)$ for all $\alpha>1$ and the system cannot be in the region $\mathcal{R}1$. Then, in $\mathcal{R}1$, we must have that $\amin>1$. 
	
	But, in this case, if there exists $\alpha\in(1,2]$ such that $G(\alpha)<\ln(b|\esp[\xi]|)$, it would hold that $G(\amin)<\ln(b|\esp[\xi]|)$, so that the system cannot be in $\mathcal{R}2$. In the same vein, if $\amin>2$, we would have that $G(2)<G(\alpha)<\ln(b|\esp[\xi]|)$, so that the system cannot be in $\mathcal{R}3$.
	
	This shows that the conditions in the theorem are mutually exclusive.
\end{remark}

\begin{remark}
\label{rk:relation-with-modulus-env}
	From the exact computations of Section \ref{sec:basic-recursions}, we will see that, in the region $\mathcal{R}1$, it holds that
	\begin{eqnarray*}
		\free(\xi)
		=
		\lim_{n\to\infty}
		\frac{1}{n}\ln |\esp[\pf_n(\xi)]|.
	\end{eqnarray*}
	The right hand side is usually known as the annealed free energy. In the nomenclature of the model with positive weights, this region corresponds to weak disorder.
	
	In Section \ref{sec:comparison} below, we will recall the possible values of the free energy for the model with positive weights. We will then be able to remark that, in the regions $\mathcal{R}2$ and $\mathcal{R}3$, it holds that
	\begin{eqnarray*}
		\free(\xi)
		=
		\lim_{n\to\infty}
		\frac{1}{2n}\ln \pf_n ( \left| \xi \right|^2 ),
	\end{eqnarray*}
	where $\pf_n ( \left| \xi \right|^2 )$ denotes the partition function of the model with weights $|\xi|^2=\left\{ \left| \xi(x) \right|^2: \ x\in\mathbb{T}\right\}$, which is in the strong disorder regime (resp. weak disorder) in the region $\mathcal{R}2$ (resp. $\mathcal{R}3$). In the case of $\mathcal{R}3$, it also holds that
	\begin{eqnarray*}
		\free(\xi)
		=
		\lim_{n\to\infty}
		\frac{1}{2n}\ln \esp \left[\pf_n ( \left| \xi \right|^2 ) \right],
	\end{eqnarray*}
	which corresponds to (half of) the annealed free energy for the model with weights $|\xi|^2$.
\end{remark}

As noted in Theorem \ref{thm:main}, the limit \eqref{eq:free-energy} is in fact an almost sure limit in the region $\mathcal{R}1$ and in the case $\amin\leq 1$. This can be extended to the whole phase diagram under an additional hypothesis on the law of the environment. Let $\tau\in(0,2]$. We say that the environment satisfies the $\tau$-condition if there exists a finite constant $C>0$ such that
\begin{eqnarray*}
	\sup_{z\in \mathbb{C}}\p[\xi \in B(z;r)] \leq C r^{\tau},\,\, \forall \, r>0,
	\quad
	\text{and}
	\quad
	\esp\left[ \left| \xi \right|^{-\tau} \right]<\infty.
\end{eqnarray*}
\begin{theorem}
\label{thm:main-as}
	In addition to the hypotheses of Theorem \ref{thm:main}, assume that the law of the environment satisfies the $\tau$-property for some $\tau\in(0,2]$. Then, the limit \eqref{eq:free-energy} holds $\p$-almost surely.
\end{theorem}
In the case of positive weights, martingale techniques are a powerful tool to obtain information on the model. This is still true for complex-valued weights, at least in the region $\mathcal{R}1$.
For $n\geq 0$, let $\T_{\leq n}=\bigcup^n_{j=1}\T_j$ and let $\calf_n = \sigma \left( \xi(x):\ x\in\T_{\leq n} \right)$.
It is a standard fact that, if $\esp[\xi]\neq 0$, then
\begin{eqnarray*}
	\mart_n(\xi)
	=
	\frac{\pf_n(\xi)}{\esp[\pf_n(\xi)]},
\end{eqnarray*}
is an $\calf_n$-martingale. More can be said in the region $\mathcal{R}1$.
\begin{proposition}
\label{thm:uniform-integrability}
The martingale $(\mart_n(\xi))_{n \geq 1}$ is uniformly integrable in the region $\mathcal{R}1$. Moreover, if
\begin{eqnarray}
	\label{eq:condition-L2}
	\esp \left[ \left| \xi \right|^2 \right]< b |\esp[\xi]|^2,
\end{eqnarray}
then
\begin{eqnarray*}
	\sup_{n \geq 1} \esp\left[ \left| \mart_n(\xi) \right|^2\right]<\infty.
\end{eqnarray*}
\end{proposition}
The condition \eqref{eq:condition-L2} defines what is known as the $\mathbb{L}^2$ region. 


\subsection{Independent radii and phases}

We now give a better picture of the phase diagram for the model with independent random phases and radii. 
Let $\{\omega(x): \, x\in\T\}$ and $\{\theta(x): \, x\in\T\}$ be two independent families of i.i.d. real-valued random variables with continuous distributions and let
\begin{eqnarray*}
	\xibg(x)=e^{\beta \omega(x)+\iota \gamma \theta(x)}.
\end{eqnarray*}
Let
\begin{eqnarray*}
	&&
	\lambda_{\R}(\beta) = \ln \esp \left[ e^{\beta \omega} \right],
	\qquad
	\lambda_{\C}(\gamma) = -\ln \left| \esp \left[ e^{\iota \gamma \theta} \right] \right|.
\end{eqnarray*}
Define the quantities $0<\beta_0<\beta_c$ and $0<\gamma_0<\gamma_c$ by the relations
\begin{align*}
	2\beta_0 \lrad'(2\beta_0)-\lrad(2\beta_0)
	&=
	\ln b,
	&
	\beta_c \lrad'(\beta_c)-\lrad(\beta_c)&=\ln b,
	\\
	\lrad(2\beta_0)-2\lrad(\beta_0)+2\lph(\gamma_0)
	&=
	\ln b,
	&
	2\lph(\gamma_c)&=\ln b.
\end{align*}
In order to keep the discussion as simple as possible, we assume that $\beta_c,\gamma_c<\infty$. Note that $\beta_0=\beta_c/2$. We denote
\begin{align*}
	\free(\beta,\gamma)
	=
	\lim_{n\to\infty}
	\frac{1}{n}\ln \left| \pf_n  \left( \xibg \right) \right|.
\end{align*}
\begin{corollary}
\label{cor:indep}
Under the hypotheses above, the three regions can be characterized as follows:
\begin{itemize}
	\item \textit{(\textsc{Region} $\mathcal{R}1$)} $\beta_0 \le \beta < \beta_c$ and $\gamma\geq0$ satisfy
	\begin{align*}
	\beta \lambda_\R' (\beta_{c}) - \lambda_\R (\beta) + \lambda_\C (\gamma) < \ln b,
	\end{align*}
	or $0 \le \beta \le \beta_0$ and $\gamma\ge0$ satisfy
	\begin{align*}
	\lambda_\R (2 \beta) - 2 \lambda_\R (\beta)+ 2 \lambda_\C (\gamma) < \ln b.
	\end{align*}
	Furthermore, 
	$ \displaystyle \free(\beta,\gamma)
	=
	\fone(\beta,\gamma)
	:=
	\ln b + \lrad(\beta)-\lph(\gamma) $.
	
	\item \textit{(\textsc{Region} $\mathcal{R}2$)} $ \beta > \beta_c $ or $\beta_0 \le \beta < \beta_c$ and $\gamma\ge0$ satisfy
	\begin{align*}
	\beta \lambda_\R' (\beta_{c}) - \lambda_\R (\beta) + \lambda_\C (\gamma) > \ln b.
	\end{align*}
	Furthermore, 
	$ \displaystyle \free(\beta,\gamma) =
	\ftwo (\beta,\gamma) :=  \beta \lrad'(\beta_c)$.

	\item \textit{(\textsc{Region} $\mathcal{R}3$)} $0 \le \beta < \beta_0$ and $\gamma\ge0$ satisfy
	\begin{align*}
	\lambda_\R (2 \beta) - 2 \lambda_\R (\beta)+ 2 \lambda_\C (\gamma) > \ln b.
	\end{align*}
	Furthermore, 
	$ \displaystyle \free(\beta,\gamma) =
	\fthree (\beta,\gamma)
	:=
	\tfrac{1}{2}\left( \ln b + \lrad(2\beta)\right)$.
\end{itemize}
%
%
%
%
\end{corollary}
Figure 1 shows the phase diagram in the case where $\omega$ and $\theta$ are independent families of i.i.d. standard Gaussian random variables.
\begin{figure}[ht!]
\label{fig:gaussian}
		\frame{\includegraphics[scale=0.10, trim=8 4 8 8, clip]{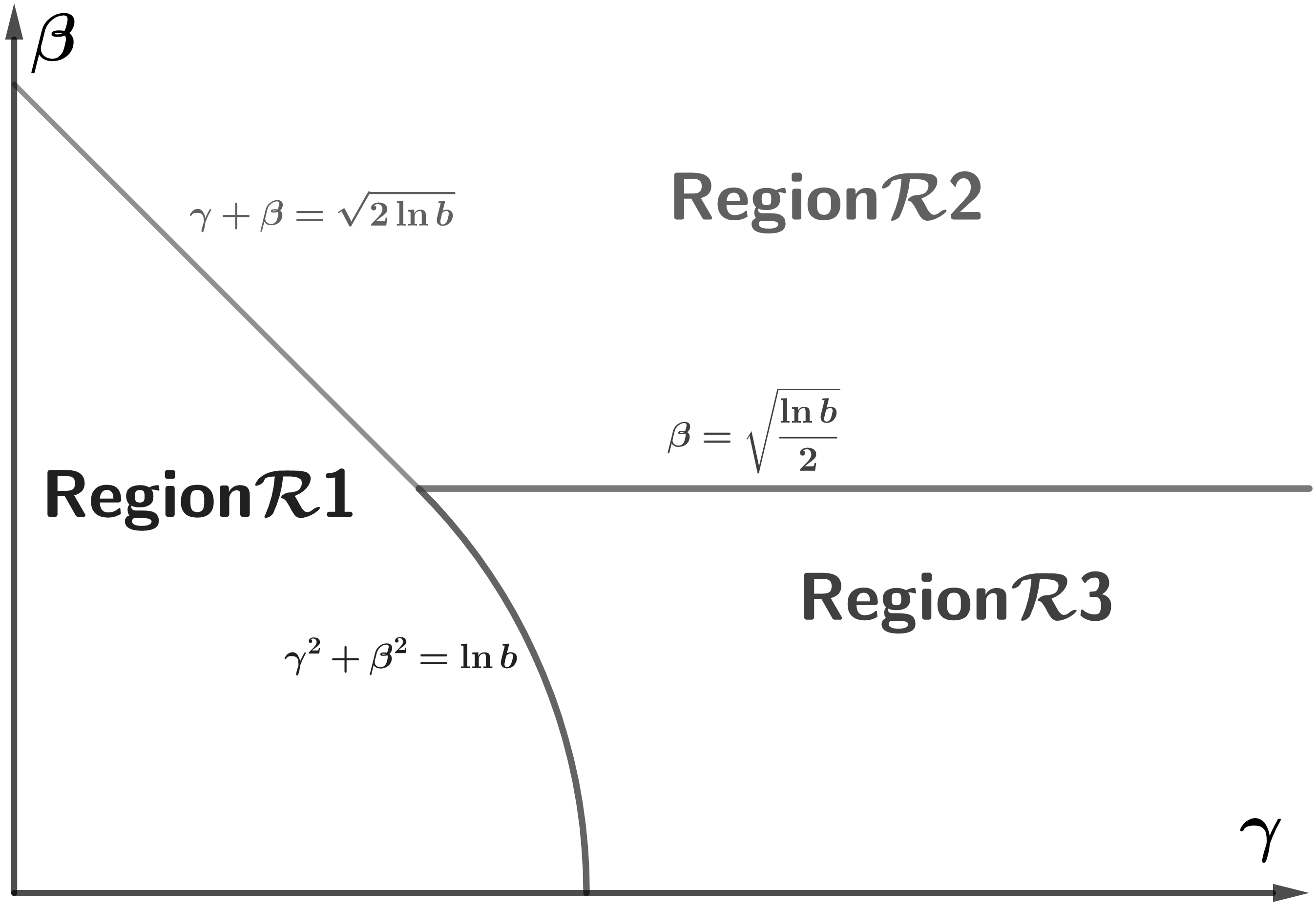}}
		\caption{Phase diagram for Gaussian environments.}
\end{figure}


\subsection{The model with positive weights}
\label{sec:comparison}

Let us recall the phase diagram of directed polymers with positive random weights on the tree.
We let $\{\omega(x):\, x\in\T\}$ be an i.i.d. family of real-valued random variables such that
\begin{eqnarray*}
	e^{\lambda(\beta)}
	=
	\esp
	\left[ e^{\beta \omega} \right]<\infty,
\end{eqnarray*}
for all $\beta>0$. Let $\eta_{\beta}(x)=e^{\beta \omega(x)}$ and
\begin{eqnarray*}
	\pf_n(\eta_{\beta})
	=
	\sum_{\walk\in\walks_n}
	\prod^n_{t=1} \eta_{\beta}(\walk_t)
	=
	\sum_{\walk\in\walks_n}
	\prod^n_{t=1} e^{\beta \omega(\walk_t)}.
\end{eqnarray*}
Let $\beta_{c}>0$ be the (unique) root of the equation
\begin{eqnarray*}
	\beta \lambda'(\beta)-\lambda(\beta)=\ln b,
\end{eqnarray*}
if it exists. Otherwise, we set $\beta_c=\infty$. The following result is taken from \cite{BPP} (see also \cite{Franchi, KP}).
\begin{theorem}
	For every $\beta\geq 0$, the limit
	\begin{eqnarray*}
		\free(\eta_{\beta})
		=
		\lim_{n\to\infty} \frac{1}{n} \ln \pf_n
		\left( \eta_{\beta} \right),
	\end{eqnarray*}
	exists $\p$-almost surely and
	\begin{eqnarray*}
		\free(\eta_{\beta})
		=
		\left\{
			\begin{array}{ll}
				\ln b + \lambda(\beta), & \text{if}\,\,\, 0\leq \beta \leq \beta_c,
				\\
				\beta \lambda'(\beta_c), & \text{if}\,\,\, \beta>\beta_c.
			\end{array}
		\right.
	\end{eqnarray*}
\end{theorem}

Let us recast all this in our language to understand the behaviour of the model with positive weights $|\xi|=\left\{ \left| \xi(x) \right|:\, x\in\mathbb{T} \right\}$. This will be instrumental in our proofs. We let
\begin{eqnarray*}
	\pf_n \left( |\xi| \right)
	=
	\sum_{\walk\in\walks_n}
	\prod^n_{t=1} |\xi(\walk_t)|,
	\qquad
	\pf_n ( \left| \xi \right|^2 )
	=
	\sum_{\walk\in\walks_n}
	\prod^n_{t=1} |\xi(\walk_t)|^2,
\end{eqnarray*}
and note that these correspond to partition functions of directed polymers in real-valued environments.
We write $|\xi(x)|=e^{\omega(x)}$, we let $\lambda(\beta)=\ln \esp \left[ \left| \xi \right|^{\beta} \right]$ be as above and we let $\beta_c$ be the critical inverse temperature for the model with weights $|\xi(x)|^{\beta}=e^{\beta \omega(x)}$. 
Hence,
\begin{eqnarray*}
	G(\alpha)
	=
	\frac{1}{\alpha} \ln\left( b\esp\left[ \left| \xi \right|^{\alpha}\right]\right)
	=
	\frac{\ln b + \lambda(\alpha)}{\alpha},
	\qquad
	G'(\alpha)
	=
	\frac{\alpha \lambda'(\alpha)-\lambda(\alpha)-\ln b}{\alpha^2}.
\end{eqnarray*}
In particular, $\amin=\beta_c$ and $G(\amin)=\lambda'(\beta_c)$. Note that the model with positive weights $|\xi(x)|^{\beta}$ is in the weak disorder regime if and only if $G'(\beta)<0$.

In the region $\mathcal{R}1$, there exists $\alpha>1$ such that $G(\alpha)<\ln(b|\esp[\xi]|)$. In particular, $G(\alpha)<\ln(b\esp[|\xi|])=G(1)$ and $G'(1)<0$. Hence, the model with positive weights $|\xi|$ is in the weak disorder regime and
\begin{eqnarray*}
		\free(|\xi|)
		=
		\lim_{n\to\infty} \frac{1}{n} \ln \pf_n(|\xi|)
		=
		\ln b + \ln \esp[|\xi|].
\end{eqnarray*}

Let us consider the region $\mathcal{R}2$.
If $\amin<1$, then $G'(1)>0$, the model with positive weights $|\xi|$ is in the strong disorder regime and
\begin{eqnarray*}
		\free(|\xi|)
		=
		\lim_{n\to\infty} \frac{1}{n} \ln \pf_n(|\xi|)
		=
		\lambda'(\beta_c)=G(\amin).
\end{eqnarray*}
The remaining piece of the region $\mathcal{R}2$ corresponds to $1\leq \amin < 2$ and $G(\amin)>\ln(b|\esp[\xi]|)$. In particular, $G'(1)<0$ and $G'(2)>0$, which means that the model with positive weights $|\xi|$ (resp. $|\xi|^2=e^{2\omega}$) is in the weak disorder regime (resp. strong disorder regime) and
\begin{align*}
		\free(|\xi|)
		&=
		\lim_{n\to\infty} \frac{1}{n} \ln \pf_n \left( \left|\xi \right| \right)
		=
		\ln b + \ln \esp[|\xi|],
		\\
		\free(|\xi|^2)
		&=
		\lim_{n\to\infty} \frac{1}{n} \ln \pf_n ( \left|\xi \right|^2 )
		=
		2G(\amin).
\end{align*}
Finally, in the region $\mathcal{R}3$, it holds that $\amin>2$ which implies that $G'(1)<0$ and $G'(2)<0$, so that the models with positive weights $|\xi|$ and $|\xi|^2$ are in the weak disorder regime and, as a consequence,
\begin{align*}
		\free(|\xi|)
		&=
		\lim_{n\to\infty} \frac{1}{n} \ln \pf_n (|\xi|)
		=
		\ln b + \ln \esp[|\xi|],
		\\
		\free ( \left|\xi \right|^2 )
		&=
		\lim_{n\to\infty} \frac{1}{n} \ln \pf_n ( \left|\xi \right|^2 )
		=
		\ln b + \ln \esp \left[ \left|\xi \right|^2 \right].
\end{align*}



\subsection{Structure of the article}

Section \ref{sec:general-estimates} contains basic estimates that will be used along the proofs. In particular, Sections \ref{sec:lower-tails} and \ref{sec:as-lower-bounds} contain crucial lower tail bounds on the partition function. In Section \ref{sec:generic}, we present the general scheme of proof used to study the regions $\mathcal{R}2$ and $\mathcal{R}3$.
Sections \ref{sec:R1}, \ref{sec:R2} and \ref{sec:R3} contain the parts of the proofs of Theorems \ref{thm:main} and \ref{thm:main-as} corresponding to the regions $\mathcal{R}1$, $\mathcal{R}2$ and $\mathcal{R}3$ respectively.
The treatment of $\mathcal{R}2$ is further divided in two cases: $\amin\leq 1$ and $1<\amin\leq 2$.
In Section \ref{sec:independent}, we specialize to independent radii and phases. Finally, the Appendix \ref{app:PZ} contains a version of the Paley-Zygmund inequality used along the proofs.


\section{General estimates}
\label{sec:general-estimates}


\subsection{Basic recursions}
\label{sec:basic-recursions}

Recall that, for $n\geq 0$, $\T_{\leq n}=\cup^n_{j=1}\T_j$ and $\calf_n = \sigma(\xi(x):\, x\in\T_{\leq n})$. 
We begin with a simple lemma.
\begin{lemma}
	\label{thm:basic-conditioning}
	For all $n\geq 1$, it holds that
	\begin{align}
		\label{eq:cond-z1}
		\esp
		\left[
		\left.
			\pf_{n+1}(\xi)
		\right|
		\calf_n
		\right]	
		&=
		b\mom_1\pf_n(\xi),
		\\
		\label{eq:cond-z2}
		\esp
		\left[
		\left.
			\left| \pf_{n+1}(\xi) \right|^2
		\right|
		\calf_n
		\right]	
		&=
		b^2 |\mom_1|^2|\pf_{n}(\xi)|^2
		+
		b \sigma^2
		\pf_n ( \left| \xi \right|^2),
	\end{align}
	where $\mom_1 = \esp[\xi]$ and $\sigma^2 = \esp [ \left| \xi \right|^2 ]-|\esp[\xi]|^2$.
\end{lemma}
\begin{proof}
	The proof of the first identity is straightforward. Next,
	\begin{align*}
	|\pf_{n+1}(\xi)|^2 
	&=
	\sum_{\walk\in\walks_{n+1}} 
		\prod^{n+1}_{t=1} |\xi(\walk_t)|^2
	+
	\sum_{\substack{\walk,\walk'\in\walks_{n+1} \\ \walk\neq\walk'}}
		\prod^{n+1}_{t=1} \xi(\walk_t) \overline{\xi(\walk'_t)}
	\\
	&=
	\pf_{n+1} ( \left| \xi \right|^2 )
	+
	\sum_{\substack{\walk\neq\walk'\in\walks_{n+1} \\ \walk_{[0,n]}=\walk'_{[0,n]}}}
		\xi(\walk_{n+1})\overline{\xi(\walk'_{n+1})}
		\prod^{n}_{t=1} |\xi(\walk_t)|^2
	\\
	&\quad
	+
	\sum_{\substack{\walk\neq\walk'\in\walks_{n+1} \\ \walk_{[0,n]}\neq\walk'_{[0,n]}}}
	\xi(\walk_{n+1})\overline{\xi(\walk'_{n+1})}
		\prod^{n}_{t=1} \xi(\walk_t) \overline{\xi(\walk'_t)},
\end{align*}
where $\walk_{[0,n]}=(\walk_0, \walk_1, \ldots, \walk_n)$. Hence,
\begin{align*}
	\esp
	\left[
		\left.
		\left| \pf_{n+1}(\xi) \right|^2 \right| \mathcal{F}_n
	\right]
	&=
	b \esp \left[ \left| \xi \right|^2 \right]
	\pf_n ( \left| \xi \right|^2 )
	+
	b(b-1)\left| \esp[\xi] \right|^2
	\sum_{\walk\in\walks_{n}}
		\prod^{n}_{t=1} |\xi(\walk_t)|^2
	\\
	&\quad
	+
	b^2\left| \esp[\xi] \right|^2
	\sum_{\walk\neq\walk'\in\walks_{n}}
		\prod^{n}_{t=1} \xi(\walk_t) \overline{\xi(\walk'_t)}
	\\
	&=
	b\left(
		\esp \left[ \left| \xi \right|^2 \right]-\left| \esp[\xi] \right|^2
	\right)
	\pf_n ( \left| \xi \right|^2 )
	+
	b^2\left| \esp[\xi] \right|^2
	|\pf_n(\xi)|^2
	\\
	&=
	b\sigma^2
	\pf_n ( \left| \xi \right|^2 )
	+
	b^2 |\mom_1|^2
	|\pf_n(\xi)|^2. \qedhere
\end{align*}
\end{proof}
Iterating \eqref{eq:cond-z1}, we easily get 
\begin{eqnarray*}
	\esp[\pf_n(\xi)]
	=
	b^n \mom_1^n.
\end{eqnarray*}
Using \eqref{eq:cond-z2}, we get the recursion
\begin{eqnarray*}
	\esp \left[ \left| \pf_{n} \right|^2 \right]
	=
	b^2|\mom_1|^2\esp \left[ \left| \pf_{n-1} \right|^2 \right]
	+
	b\sigma^2 \esp \left[ \pf_{n-1} ( \left| \xi \right|^2 ) \right]
	=
	b^2|\mom_1|^2\esp \left[ \left| \pf_{n-1} \right|^2 \right]
	+
	b^{n}\sigma^2 \esp \left[ \left| \xi \right|^2 \right]^{n-1}.
\end{eqnarray*}
Let $\bmom_2=\esp \left[ \left| \xi \right|^2 \right]$.
Inductively,
\begin{eqnarray*}
	\esp \left[ \left| \pf_{n} \right|^2 \right]
	&=&
	b^n \left( b \left| \mom_1 \right|^{2} \right)^n
	+
	\sigma^2
	b^n \bmom_2^{n-1}
	\sum^{n-1}_{j=0}
	\left(
		\frac{b \left| \mom_1 \right|^2}{\bmom_2}
	\right)^j.
\end{eqnarray*}
If $b|\mom_1|^2 < \bmom_2$, the sum is bounded and we have 
\begin{eqnarray*}
	\lim_{n\to\infty}
	\frac{1}{n}\ln \esp \left[ \left| \pf_n(\xi) \right|^2 \right]
	=
	\ln\left(
		b \bmom_2
	\right).
\end{eqnarray*}
Now, if $b|\mom_1|^2>\bmom_2$, the sum growths as $(b |\mom_1|^2)^n$ and
\begin{eqnarray*}
	\lim_{n\to\infty}
	\frac{1}{n}\ln \esp \left[ \left| \pf_n(\xi) \right|^2 \right]
	=
	\ln\left(
		b^2 \left| \mom_1 \right|^2
	\right).
\end{eqnarray*}
Finally, if $b|\mom_1|^2=\bmom_2$, the sum growths linearly and
\begin{eqnarray*}
	\lim_{n\to\infty}
	\frac{1}{n}\ln \esp \left[ \left| \pf_n(\xi) \right|^2 \right]
	=
	\ln\left(
		b \bmom_2
	\right)
	=
	\ln\left(
		b^2 \left| \mom_1 \right|^2
	\right).
\end{eqnarray*}
Let us summarize our discussion.
\begin{lemma}
	If $b|\esp[\xi]|^2>\esp \left[ \left| \xi \right|^2 \right]$,
	\begin{eqnarray*}
		\lim_{n\to\infty}
		\frac{1}{n}\ln \esp \left[ \left| \pf_n(\xi) \right|^2 \right]
		=
		\ln\left(
			b^2 \left| \esp[\xi] \right|^2
		\right).
	\end{eqnarray*}
	If $b|\esp[\xi]|^2<\esp \left[ \left| \xi \right|^2 \right]$,
	\begin{eqnarray*}
		\lim_{n\to\infty}
		\frac{1}{n}\ln \esp \left[ \left| \pf_n(\xi) \right|^2 \right]
		=
		\ln\left(
			b \esp \left[ \left| \xi \right|^2 \right]
		\right).
	\end{eqnarray*}
	If $b|\esp[\xi]|^2=\esp \left[ \left| \xi \right|^2 \right]$,
	\begin{eqnarray*}
		\lim_{n\to\infty}
		\frac{1}{n}\ln \esp \left[ \left| \pf_n(\xi) \right|^2 \right]
		=
		\ln\left(
			b^2 \left| \esp[\xi] \right|^2
		\right)
		=
		\ln\left(
			b \esp \left[ \left| \xi \right|^2 \right]
		\right).
	\end{eqnarray*}
\end{lemma}
Note that, if $\esp[\xi]=0$, only the second possibility can occur.


\subsection{Moments estimates}

Recall that
\begin{eqnarray*}
	G(\alpha)
	=
	\frac{1}{\alpha} 
	\ln\left(
		b\esp[|\xi|^{\alpha}]
	\right),
	\quad
	\alpha>0.
\end{eqnarray*}
We can easily verify that either $G$ has a unique critical point $\alpha_{\textsc{min}}$ (which, in this case, corresponds to the global minimizer) or $G$ is strictly decreasing.

Assume that $\esp[\xi]\neq 0$.
Let
\begin{eqnarray*}
	\mart_n(\xi)
	=
	\frac{\pf_n(\xi)}{\esp[\pf_n(\xi)]},
	\quad\quad
	\martx_n(\xi)
	=
	\frac{|\pf_n(\xi)|^2}{\esp \left[ \left| \pf_n(\xi) \right|^2 \right]}.
\end{eqnarray*}
\begin{lemma}
	\label{thm:moment-estimates}
	If $b|\esp[\xi]|^2>\esp \left[ \left| \xi \right|^2 \right]$, i.e., if $G(2)<\ln(b|\esp[\xi]|)$, then
	\begin{eqnarray*}
		\sup_{{n\geq 1}}
		\esp\left[
			\left|
				\mart_n(\xi)
			\right|^{\theta}
		\right]
		<\infty,
	\end{eqnarray*}	
	for all $\theta\in[0,2]$.
	
	If there exists $\theta\in(1,2)$ such that $G(\theta)<G(2)$, then
	\begin{eqnarray*}
		\sup_{{n\geq 1}}
		\esp\left[
			\left|
				\mart_n(\xi)
			\right|^{\theta}
		\right]
		<\infty.
	\end{eqnarray*}	
	If  $\amin>2$, there exists $\nu>1$ such that
	\begin{eqnarray*}
		\sup_{{n\geq 1}}
		\esp
		\left[
			\left|
				\martx_n(\xi)
			\right|^{\nu}
		\right]
		<\infty.
	\end{eqnarray*}	
\end{lemma}
\begin{proof}
	Let us prove the first estimate. Let $\theta\in(0,2)$. Then,
	\begin{align*}
		\esp\left[
			\left|
				\pf_n(\xi)
			\right|^{\theta}
		\right]
		&=
		\esp\left[
			\esp
			\left[
			\left.
			\left|
				\pf_n(\xi)
			\right|^{2 \ \frac{\theta}{2}}
			\right|
			\calf_{n-1}
			\right]
		\right]
		\le
		\esp\left[
			\esp
			\left[
			\left.
			\left|
				\pf_n(\xi)
			\right|^{2 }
			\right|
			\calf_{n-1}
			\right]^{\frac{\theta}{2}}
		\right],
	\end{align*}
	by Jensen's inequality (as $\frac{\theta}{2}\in(0,1)$). Now,
	\begin{eqnarray*}
		\esp\left[
			\left.
			\left|
				\pf_n(\xi)
			\right|^{2 }
			\right|
			\calf_{n-1}
		\right]^{\frac{\theta}{2}}
		&=&
		\left[
			b\sigma^2 \pf_{n-1} ( \left| \xi \right|^2 ) + b^2 |\mom_1|^2 |\pf_{n-1}(\xi)|^2
		\right]^{\frac{\theta}{2}}
		\\
		&\leq&
		b^{\frac{\theta}{2}}\sigma^{\theta}\pf_{n-1} ( \left| \xi \right|^2 )^{\frac{\theta}{2}}
		+
		b^{\theta}|\mom_1|^{\theta}|\pf_{n-1}(\xi)|^{\theta}.
	\end{eqnarray*}
	Integrating,
	\begin{eqnarray*}
		\esp\left[
			\left|
				\pf_n(\xi)
			\right|^{\theta}
		\right]
		&\leq&
		b^{\frac{\theta}{2}}\sigma^{\theta}
		\esp\left[
			\pf_{n-1} ( \left| \xi \right|^2 )^{\frac{\theta}{2}}
		\right]
		+
		b^{\theta}|\mom_1|^{\theta}
		\esp\left[
			|\pf_{n-1}(\xi)|^{\theta}
		\right].
	\end{eqnarray*}
	With the normalization $|\esp[\pf_n(\xi)]|=b^{\theta n} |\mom_1|^{\theta n}$,
	\begin{eqnarray*}
		\esp\left[
			\left|
				\mart_n(\xi)
			\right|^{\theta}
		\right]
		&\leq&
		\frac{b^{\frac{\theta}{2}}\sigma^{\theta}}{b^{\theta n} |\mom_1|^{\theta n}}
		\esp\left[
			\pf_{n-1} ( \left| \xi \right|^2 )^{\frac{\theta}{2}}
		\right]
		+
		\esp\left[
			\left|
				\mart_{n-1}(\xi)
			\right|^{\theta}
		\right]
		\\
		&=&
		b^{\frac{\theta}{2}}\sigma^{\theta}
		\frac{b^n \bmom_2^n}{b^{\theta n} |\mom_1|^{\theta n}}
		\esp\left[
			\mart_{n-1} ( \left| \xi \right|^2 )^{\frac{\theta}{2}}
		\right]
		+
		\esp\left[
			\left|
				\mart_{n-1}(\xi)
			\right|^{\theta}
		\right]
		\\
		&\leq&
		b^{\frac{\theta}{2}}\sigma^{\theta}
		\frac{b^{n\frac{\theta}{2}} \bmom_2^{n\frac{\theta}{2}}}{b^{\theta n} |\mom_1|^{\theta n}}
		\esp\left[
			\mart_{n-1} ( \left| \xi \right|^2 )
		\right]^{\frac{\theta}{2}}
		+
		\esp\left[
			\left|
				\mart_{n-1}(\xi)
			\right|^{\theta}
		\right]
		\\
		&=&
		b^{\frac{\theta}{2}}\sigma^{\theta}
		\left(
			\frac{\bmom_2^{1/2}}{b^{1/2} |\mom_1|}
		\right)^{n\theta}
		+
		\esp\left[
			\left|
				\mart_{n-1}(\xi)
			\right|^{\theta}
		\right],
	\end{eqnarray*}
	by Jensen's inequality. The estimate will then hold if $\bmom < b \mom_1^2$
	which is our hypothesis. 
	
	Let us prove the second estimate. Let $\theta\in(0,2)$. As above,
	\begin{eqnarray*}
		\esp\left[
			\left|
				\pf_n(\xi)
			\right|^{\theta}
		\right]
		&\leq&
		\esp\left[
			\esp
			\left[
			\left.
			\left|
				\pf_n(\xi)
			\right|^{2 }
			\right|
			\calf_{n-1}
			\right]^{\frac{\theta}{2}}
		\right].
	\end{eqnarray*}
	Now,
	\begin{eqnarray*}
		&&
		\esp\left[
			\left.
			\left|
				\pf_n(\xi)
			\right|^{2 }
			\right|
			\calf_{n-1}
		\right]^{\frac{\theta}{2}}
		\\
		&&
		=
			\esp\left[
			\left.
				\sum_{x\in\T_1} |\xi(x)|^2 |\pf_{n-1,x}(\xi)|^2
				+
				\sum_{\substack{x,x'\in\T_1\\ x\neq x'}}
				\xi(x)\pf_{n-1,x}(\xi) \ \overline{ \xi \left( x' \right) \pf_{n-1,x} \left( \xi' \right)}
			\right| \calf_{n-1}
			\right]^{\frac{\theta}{2}}
		\\
		&&
		=
		\left(
		\sum_{x\in\T_1} 
		|\xi(x)|^2
		\esp \left[ \left. \left| \pf_{n-1,x}(\xi) \right|^2  \right|\calf_{n-1} \right]
		+
		\sum_{\substack{x,x'\in\T_1\\ x\neq x'}}
		\xi(x)\overline{\xi \left( x' \right)}
		\esp \left[ \left. \pf_{n-1,x}(\xi) \right|\calf_{n-1} \right]
		\esp \left[ \left. \overline{\pf_{n-1,x} \left( \xi' \right)} \right|\calf_{n-1} \right]
		\right)^{\frac{\theta}{2}}
		\\
		&&
		\leq
		\sum_{x\in\T_1} 
		|\xi(x)|^{\theta}
		\esp \left[ \left. \left| \pf_{n-1,x}(\xi) \right|^2  \right| \calf_{n-1} \right]^{\frac{\theta}{2}} 
		+
		\left(
		\sum_{\substack{x,x'\in\T_1\\ x\neq x' }}
				\xi(x)\overline{\xi \left( x' \right) }
		\esp \left[ \left. \pf_{n-1,x}(\xi) \right| \calf_{n-1} \right]
		\esp \left[ \left. \overline{\pf_{n-1,x} \left( \xi' \right)} \right| \calf_{n-1} \right]
		\right)^{\frac{\theta}{2}} \!\!\!.
	\end{eqnarray*}
	Integrating,
	\begin{align*}
		\esp\left[
			\esp\left[
			\left|
				\pf_n(\xi)
			\right|^{2 }
			\Big{|} \calf_{n-1}
			\right]^\frac{\theta}{2}
		\right]
		&\leq
		\sum_{x\in\T_1} 
		\esp \left[ \left| \xi(x) \right|^{\theta} \right]
		\esp \left[ \esp \left[ \left. \left| \pf_{n-1,x}(\xi) \right|^2  \right| \calf_{n-1} \right]^{\frac{\theta}{2}} \right]
		\\
		&
		\quad +
		\esp\left[
		\left(
		\sum_{\substack{x,x'\in\T_1\\ x\neq x' }}
				\xi(x)\overline{\xi(x')}
		\esp[\pf_{n-1,x}(\xi) |\calf_{n-1}]
		\esp \left[ \left. \overline{\pf_{n-1,x}(\xi')} \right| \calf_{n-1} \right]
		\right)^{\frac{\theta}{2}}
		\right]
		\\
		&\leq
		b
		\esp \left[ \left| \xi \right|^{\theta} \right]
		\esp \left[ \esp \left[ \left. \left| \pf_{n-1}(\xi) \right|^2  \right| \calf_{n-2} \right]^{\frac{\theta}{2}} \right]
		\\
		&\quad +
		\esp\left[
		\sum_{\substack{ x,x'\in\T_1 \\ x\neq x' }}
		\xi(x)\overline{\xi(x')}
		\esp \left[ \left. \pf_{n-1,x}(\xi) \right| \calf_{n-1} \right]
		\esp \left[ \left. \overline{\pf_{n-1,x}(\xi')} \right| \calf_{n-1} \right]
		\right]^{\frac{\theta}{2}}
		\\
		&
		\leq
		b
		\esp \left[ \left| \xi \right|^{\theta} \right]
		\esp \left[ \esp \left[ \left. \left| \pf_{n-1}(\xi) \right|^2  \right| \calf_{n-2} \right]^{\frac{\theta}{2}} \right]
		\\
		& \quad +
		\left(
		\sum_{\substack{x,x'\in\T_1\\ x\neq x' }}
		|\esp[\xi]|^2
		\esp \left[\esp \left[ \left. \pf_{n-1,x}(\xi) \right| \calf_{n-1} \right] \right]
		\,
		\esp \left[ \esp \left[ \left. \overline{\pf_{n-1,x}(\xi')} \right| \calf_{n-1} \right] \right] 
		\right)^{\frac{\theta}{2}}
		\\
		& \leq
		b
		\esp \left[ \left| \xi \right|^{\theta} \right]
		\esp \left[ \esp \left[ \left. \left| \pf_{n-1}(\xi) \right|^2  \right| \calf_{n-2} \right]^{\frac{\theta}{2}} \right]
		+
		\left(
		\sum_{\substack{x,x'\in\T_1\\ x\neq x' }}
				|\esp[\xi]|^2
		\left| \esp[\pf_{n-1,x}(\xi)]\right|^2
		\right)^{\frac{\theta}{2}}
		\\
		& \leq
		b
		\esp \left[ \left| \xi \right|^{\theta} \right]
		\esp \left[ \esp \left[ \left. \left| \pf_{n-1}(\xi) \right|^2  \right| \calf_{n-2} \right]^{\frac{\theta}{2}} \right]
		+
		b^{\frac{\theta}{2}}(b-1)^{\frac{\theta}{2}}
				|\esp[\xi]|^\theta
		\left| \esp[\pf_{n-1,x}(\xi)]\right|^{\theta}.
	\end{align*}
	Summarizing, 
	\begin{align*}
	\nonumber
		\esp\left[
			\esp
			\left[
				\left.
					\left|
						\pf_n(\xi)
					\right|^{2 }
				\right|
				\calf_{n-1}
			\right]^{\frac{\theta}{2}}
		\right]
		&\leq
		b
		\esp \left[ \left| \xi \right|^{\theta} \right]
		\,
		\esp \left[ \esp \left[ \left. \left| \pf_{n-1}(\xi) \right|^2  \right|\calf_{n-2}\right]^{\frac{\theta}{2}} \right]
		+
		b^{\frac{\theta}{2}}(b-1)^{\frac{\theta}{2}}
				|\esp[\xi]|^\theta
		\left| \esp[\pf_{n-1}(\xi)] \right|^{\theta}
		\\
		&=
		b
		\esp \left[ \left| \xi \right|^{\theta} \right]
		\,
		\esp \left[ \esp \left[ \left. \left| \pf_{n-1}(\xi) \right|^2  \right|\calf_{n-2}\right]^{\frac{\theta}{2}} \right]
		+
		b^{\frac{\theta}{2}}(b-1)^{\frac{\theta}{2}} b^{\theta n}|\esp[\xi]|^{\theta n}.
	\end{align*}
	Now, let
	\begin{eqnarray}
	\label{eq:aux-moments}
		R_n
		:=
		\esp
		\left[
			\esp
			\left[
			\left.
			\left|
				\mart_n(\xi)
			\right|^{2 }
			\right| \calf_{n-1}
			\right]^{\frac{\theta}{2}}
		\right]
		=
		\frac{1}{(b|\esp[\xi]|)^{n\theta}}
		\esp\left[
			\esp
			\left[
			\left.
			\left|
				\pf_n(\xi)
			\right|^{2 }
			\right|
			\calf_{n-1}
			\right]^{\frac{\theta}{2}}
		\right].
	\end{eqnarray}
	Then,
	\begin{eqnarray*}
		R_n
		&\leq&
		\frac{b\esp \left[ \left| \xi \right|^{\theta} \right]}{ \left(b \left| \esp[\xi] \right| \right)^{\theta}}
		\,
		R_{n-1}
		+
		b^{-\frac{\theta}{2}}(b-1)^{\frac{\theta}{2}}.
	\end{eqnarray*}
	The sequence $(R_n)_{n \ge 1}$ will be bounded if
	\begin{eqnarray*}
		\frac{b\esp \left[ \left| \xi \right|^{\theta} \right]}{ \left(b \left| \esp[\xi] \right| \right)^{\theta}} < 1,
	\end{eqnarray*}
	which is equivalent to $\displaystyle b\esp \left[ \left| \xi \right|^{\theta} \right]<(b|\esp[\xi]|)^{\theta}$, i.e., $\displaystyle G(\theta)<\ln\left(b|\esp[\xi]| \right)$.
	
	 Next, we prove the third estimate. We assume that $\amin>2$. Let $\theta\in(2,8/3)$ so that $\theta/4<1$.
	 Then,
	 \begin{align*}
	 	\esp\left[
	 		\left| \pf_{n+1}\right|^{\theta}
	 	\right]
	 	&=
	 	\esp\left[
	 		\left| \sum_{x\in\mathbb{T}_1} \xi(x)\pf_{n,x} \right|^{\theta}
	 	\right]
	 	\leq
	 	\esp\left[
	 		\left( \sum_{x\in\mathbb{T}_1} |\xi(x)|\, |\pf_{n,x}| \right)^{4 \ \frac{\theta}{4}}
	 	\right]
	 	\leq
	 	\esp\left[
	 		\left( \sum_{x\in\mathbb{T}_1} |\xi(x)|^{\frac{\theta}{4}}\, |\pf_{n,x}|^{\frac{\theta}{4}} \right)^{4}
	 	\right]
	 	\\
	 	&=
	 	b \esp\left[ |\xi|^{\theta} \right]\esp\left[ |\pf_{n}|^{\theta}  \right]
	 	\\
	 	&\quad +
	 	3b(b-1)
	 	\esp\left[ |\xi|^{\frac{\theta}{2}}\right]^2
	 	\esp\left[  |\pf_{n}|^{\frac{\theta}{2}} \right]^2
	 	+
	 	4b(b-1)
	 	\esp\left[ |\xi|^{\frac{\theta}{4}}\right]
	 	\esp\left[  |\pf_{n}|^{\frac{\theta}{4}} \right]
	 	\esp\left[ |\xi|^{\frac{3\theta}{4}}\right]
	 	\esp\left[  |\pf_{n}|^{\frac{3\theta}{4}} \right]
	 	\\
	 	&\quad
	 	+
	 	6b(b-1)(b-2)
	 	\esp\left[ |\xi|^{\frac{\theta}{4}}\right]^2
	 	\esp\left[  |\pf_{n}|^{\frac{\theta}{4}} \right]^2
	 	\esp\left[ |\xi|^{\frac{\theta}{2}}\right]
	 	\esp\left[  |\pf_{n}|^{\frac{\theta}{2}} \right]
	 	\\
	 	&\quad
	 	+
	 	b(b-1)(b-2)(b-3)
	 	\esp\left[ |\xi|^{\frac{\theta}{4}}\right]^4
	 	\esp\left[  |\pf_{n}|^{\frac{\theta}{4}} \right]^4
	 	\\
	 	&\leq
	 	b \bmom_{\theta}
	 	\esp \left[ \left| \pf_n \right|^2 \right]^{\frac{\theta}{2}}
	 	+
	 	3b(b-1)
	 	\bmom_2^{\frac{\theta}{2}}
	 	\esp \left[ \left| \pf_n \right|^2 \right]^{\frac{\theta}{2}}
	 	+
	 	4b(b-1)
	 	\bmom_2^{\frac{\theta}{8}+\frac{3\theta}{8}}
	 	\esp \left[ \left| \pf_n \right|^2 \right]^{\frac{\theta}{8}+\frac{3\theta}{8}}
	 	\\
	 	&\quad
	 	+
	 	6b(b-1)(b-2)
	 	\bmom_2^{\frac{\theta}{4}+\frac{3\theta}{4}}
	 	\esp \left[ \left| \pf_n \right|^2 \right]^{\frac{\theta}{4}+\frac{3\theta}{4}}
	 	+
	 	b(b-1)(b-2)(b-3)
	 	\bmom_2^{\frac{\theta}{2}}
	 	\esp \left[ \left| \pf_n \right|^2 \right]^{\frac{\theta}{2}}
	 	\\
	 	&=
	 	b \bmom_{\theta}\esp\left[ |\pf_{n}|^{\theta}  \right]
	 	+
	 	b(b-1)(b^2+3)\bmom_2^{\frac{\theta}{2}}
	 	\esp \left[ \left| \pf_n \right|^2 \right]^{\frac{\theta}{2}},
	 \end{align*}
	where \(\bmom_\theta = \esp \left[ \left| \xi \right|^\theta \right]\). Now, as
	\begin{align*}
		|\pf_{n+1}|^2
		=
		\sum_{x\in\mathbb{T}_1}|\xi(x)|^2|\pf_{n,x}|^2
		+
		\sum_{\substack{x,x'\in\mathbb{T}_1 \\ x\neq x'}}
			\xi(x)\pf_{n,x}\overline{\xi(x')\pf_{n,x'}},
	\end{align*}
	we obtain
	\begin{align*}
		\esp \left[ \left| \pf_{n+1} \right|^2 \right]
		=
		b \bmom_2 \esp \left[ \left| \pf_n \right|^2 \right]
		+
		b(b-1)|\esp[\xi]|^2|\esp[\pf_n]|^2,
	\end{align*}
	so that $\esp \left[ \left| \pf_{n+1} \right|^2 \right]\geq b \bmom_2 \esp \left[ \left| \pf_n \right|^2 \right]$.
	 Hence,
	 \begin{align*}
	 	\frac{
	 		\esp\left[
	 			\left| \pf_{n+1}\right|^{\theta}
	 		\right]
	 	}{
					\esp\left[
	 					\left| \pf_{n+1}\right|^{2}
	 				\right]^{\frac{\theta}{2}} 	
	 	}
	 	\leq
	 	\frac{b \bmom_{\theta}}{(b\bmom_2)^{\frac{\theta}{2}}}
	 	\frac{
	 		\esp\left[
	 			\left| \pf_{n}\right|^{\theta}
	 		\right]
	 	}{
					\esp\left[
	 					\left| \pf_{n}\right|^{2}
	 				\right]^{\frac{\theta}{2}} 	
	 	}
	 	+
	 	\frac{b(b-1)(b^2+3)}{b^{\frac{\theta}{2}}}.
	 \end{align*}
	 This stays bounded if $\displaystyle \frac{b \bmom_{\theta}}{(b\bmom_2)^{\frac{\theta}{2}}}<1$, i.e.,
	 if $G(\theta)<G(2)$. As we assumed that $\amin>2$, this is satisfied for $\theta>2$ small enough. We then take $\nu=\theta/2$. \qedhere

\end{proof}

We can now obtain the moment estimates on the martingale $(\mart_n(\xi))_{n \ge 1}$.
\begin{proof}[Proof of Proposition \ref{thm:moment-estimates}]
	The first statement is an immediate consequence of the discussion above. To show the $\mathbb{L}^2$ condition, we start from the recursion
\begin{eqnarray*}
		\esp
		\left[
		\left. 
			|\pf_{n+1}(\xi)|^2
		\right|
		\calf_n
		\right]	
		&=&
		b^2 |\mom_1|^2|\pf_{n}(\xi)|^2
		+
		b \sigma^2
		\pf_n \left( \left| \xi \right|^2 \right).
\end{eqnarray*}
Then,
\begin{eqnarray*}
		\esp\left[
			|\mart_{n+1}(\xi)|^2\Big{|}\calf_n
		\right]	
		&=&
		b^2 |\mom_1|^2
		\left(
			b|\mom_1|
		\right)^{-2(n+1)}
		|\pf_{n}(\xi)|^2
		+
		b \sigma^2
		\left(
			b|\mom_1|
		\right)^{-2(n+1)}
		\pf_n \left( \left| \xi \right|^2 \right)
		\\
		&=&
		|\mart_{n}(\xi)|^2
		+
		\frac{\sigma^2}{b \left| \mom_1 \right|^2}
		\left(
			\frac{\bmom_2}{b \left| \mom_1 \right|^2}
		\right)^n
		\mart_n \left( \left| \xi \right|^2 \right).
\end{eqnarray*}
The second moments will then be bounded as long as $\bmom_2 < b |\mom_1|^2$, i.e., if
\begin{align*}
	\esp \left[ \left| \xi \right|^2 \right] &< b |\esp[\xi]|^2.
	\qedhere
\end{align*}
\end{proof}

 \subsection{Lower tail bounds}
 \label{sec:lower-tails}
The following results correspond to Lemma 6.3 and Theorem 6.4 from \cite{DES-Complex}. We include the proofs for the convenience of the reader.
\begin{lemma}
	Suppose that there exists $n\geq 1$, $\bigfree\in\R$ and $\eta\in (0,1)$ such that
	\begin{eqnarray}
	\label{eq:lower-bound-input}
		\p\left[
			\ln |\pf_n(\xi)| < \bigfree \,
		\right]
		<
		\eta.
	\end{eqnarray}
	Then, for all $\varepsilon>0$, there exists $k_0=k_0(\varepsilon,\eta)\geq 1$ and $\delta_0=\delta_0(\varepsilon,\eta)>0$, both independent of $n$, such that
	\begin{eqnarray*}
		\p\left[
			\ln \left| \pf_{n+k}(\xi) \right| < \bigfree - \delta k\,
		\right]
		<
		\varepsilon,
	\end{eqnarray*}
	for all $k\geq k_0$ and all $\delta \geq \delta_0$.
\end{lemma}
\begin{proof}
	Let us drop $\xi$ from the notation. We decompose
	\begin{eqnarray*}
		\pf_{n+1}
		=
		\sum_{x\in\mathbb{T}_1} \xi(x) \pf_{n,x},
	\end{eqnarray*}
	and we let $\pf_n^*=\max_{x\in\mathbb{T}_1} |\pf_{n,x}|$.
	Then,
	\begin{eqnarray*}
		\ln |\pf_{n+1}|
		=
		\ln \left| \sum_{x\in\mathbb{T}_1} \xi(x) \frac{\pf_{n,x}}{\pf_n^*}\right|
		+
		\ln |\pf_n^*|
		=:
		\ln A_n
		+
		\ln |\pf_n^*|.
	\end{eqnarray*}
	Let
	\begin{eqnarray*}
		B 
		= 
		\left\{
			\ln \left| \pf_n^* \right| < \bigfree
		\right\}.
	\end{eqnarray*}
	Hence,
	\begin{eqnarray*}
		\p\left[ B \right] < \eta^b.
	\end{eqnarray*}
	Let $\delta>0$. Then,
	\begin{eqnarray*}
		\p\left[
			\ln |\pf_{n+1}| < \bigfree - \delta\,
		\right]
		&=&
		\p\left[
			\ln |\pf_{n+1}| < \bigfree - \delta,\, B
		\right]
		+
		\p\left[
			\ln |\pf_{n+1}| < \bigfree - \delta, \, B^c
		\right]
		\\
		&\leq&
		\eta^b
		+
		\p\left[
			\ln A_n+ \ln |\pf_{n}^*| < \bigfree - \delta, \, \ln |\pf_{n}^*| \geq \bigfree
		\right]
		\\
		&=&
		\eta^b
		+
		\p\left[
			\ln A_n  < - \delta
		\right].
	\end{eqnarray*}
	Let $\nu>0$ and let $\kappa>0$ be such that
	\begin{eqnarray*}
		\p[\xi \in B(z;\kappa)]<\nu,
	\end{eqnarray*}
	for all $z\in\mathbb{C}$. Then, if $x_*\in\mathbb{T}_1$ is defined such that $\pf_{n,x_*}=\pf^*_n$ and $x\in\mathbb{T}_1$, there exists $z(\xi)\in\mathbb{C}$ independent of $\xi(x)$ such that
	\begin{eqnarray*}
		\p\left[ \left. A_n < \kappa \right| \, x_*=x \right]
		= 
		\p[\xi(x) \in B(z(\xi);\kappa)]
		< \nu.
	\end{eqnarray*}
	It is easy to see that the same bound holds without conditioning.
	Hence, if $e^{-\delta}<\kappa$,
	\begin{eqnarray*}
		\p\left[
			\ln |\pf_{n+1}| < \bigfree - \delta\,
		\right]
		<
		\eta^b + \nu =: \varepsilon_1.
	\end{eqnarray*}
	Repeating this argument using this last bound as an input instead of \eqref{eq:lower-bound-input}, we get 
	\begin{eqnarray*}
		\p\left[
			\ln |\pf_{n+2}| < \bigfree - 2\delta\,
		\right]
		<
		\varepsilon_1^b + \nu =:\varepsilon_2,
	\end{eqnarray*}
	and, inductively,
	\begin{eqnarray*}
		\p\left[
			\ln |\pf_{n+k}| < \bigfree - k\delta\,
		\right]
		<
		\varepsilon_k,
	\end{eqnarray*}
	where
	\begin{eqnarray*}
		\varepsilon_k = \varepsilon_{k-1}^b + \nu.
	\end{eqnarray*}
	Let $\varepsilon>0$ and $f_{\nu}(x)=x^b+\nu$, $x\in[0,1]$. Then, there exists $\nu_0>0$ such that, if $0<\nu<\nu_0$, then
	\begin{itemize}
		\item $f_{\nu}$ has two fixed points $0<x_{\nu}<y_{\nu}<1$,
		
		\item $x_{\nu}$ is stable,
		
		\item $x_{\nu}<\varepsilon$,
		
		\item $\eta < y_{\nu}$.
	\end{itemize}
	Hence, if we choose $0<\nu<\nu_0$, there exists $k_0=k_0(\varepsilon,\eta) \geq 1$ such that $\varepsilon_k < \varepsilon$, for all $k \geq k_0$.
\end{proof}

\begin{corollary}
	\label{thm:lower-tail}
	Let $c>0$.
	Suppose that there exists $n_0\geq 1$, $\free\in\R$ and $\eta\in (0,1)$ such that
	\begin{eqnarray*}
		\p\left[
			\frac{1}{n}\ln |\pf_n(\xi)| < \free - c \,
		\right]
		<
		\eta,
	\end{eqnarray*}
	for all $n\geq n_0$.
	Then, for all $\varepsilon>0$, there exists $n_1=n_1(\varepsilon,\eta, c, n_0) \geq 1$ such that
	\begin{eqnarray*}
		\p\left[
			\frac{1}{n}\ln |\pf_n(\xi)| < \free - 2c \,
		\right]
		<
		\varepsilon,
	\end{eqnarray*}
	for all $n\geq n_1$.
\end{corollary}
\begin{proof}
	Let $n\geq n_0$. Then, by hypothesis, if $\bigfree = n(\free-c)$, we have 
	\begin{eqnarray*}
		\p\left[
			\ln |\pf_n(\xi)| < \bigfree \,
		\right]
		<
		\eta.
	\end{eqnarray*}
	Let $\delta_0>0$ and $k_0\geq 1$ be as in the previous lemma, and fix $\delta \geq \delta_0$, $k\geq k_0$.
	Hence,
	\begin{eqnarray*}
		\p\left[
			\ln |\pf_{n+k}(\xi)| < \bigfree - \delta k\,
		\right]
		<
		\varepsilon.
	\end{eqnarray*}
	But, now, 
	\begin{align*}
		\frac{1}{n+k}\left( \bigfree - \delta k \right)
		=
		\frac{1}{n+k}\left( n \free - nc - \delta k \right)
		=
		\free
		-
		c
		-
		\frac{k}{n+k} \left( \free + c + \delta \right)
		\geq
		\free - 2c,
	\end{align*}
	if $n \geq n_1$, for some $n_1\geq 1$ large enough. Hence, if $n \geq n_1$, 
	\begin{align*}
		\p\left[
			\frac{1}{n+k}
			\ln |\pf_{n+k}(\xi)| < \free - 2c\,
		\right]
		&<
		\varepsilon. \qedhere
	\end{align*}
\end{proof}

\subsection{Almost-sure lower bounds}
\label{sec:as-lower-bounds}

The following result improves the lemmas of the previous section under the $\tau$-condition and will be the key to show the almost sure convergence to the free energy.
\begin{lemma}
\label{thm:as-lower-bound}
	Under the hypothesis of the previous lemma, assume that, for all $c>0$, there exists $n_0\geq 1$ such that
	\begin{eqnarray*}
		\p\left[
			\frac{1}{n} \ln |\pf_n(\xi)|< \free - c
		\right]
		<
		1,
	\end{eqnarray*}
	for all $n\geq n_0$. Then,
	\begin{eqnarray*}
		\liminf_{n \to \infty} \frac{1}{n} \ln |\pf_n(\xi)|
		\geq
		\free,
		\qquad
		\p-\text{almost surely}.
	\end{eqnarray*}
\end{lemma}
\begin{proof}
	Let $c>0$. 
	Let $m\geq 1$ to be fixed later and $l\in\{0,\cdots,m\}$. We decompose
	\begin{eqnarray*}
		\pf_{l+(m+1)k}
		=
		\sum_{x\in\mathbb{T}_{l+k}}
			\xi_{l+k}(x) \pf^{l+k}_{mk,x},
	\end{eqnarray*}
	where, if we identify $x\in\mathbb{T}_{l+k}$ with the path $(x_0,\cdots,x_{k+l})\in\walks_{k+l}$ with $x_{k+l}=x$,
	\begin{eqnarray*}
		\xi_{l+k}(x) = \prod^{k+l}_{t=1}\xi(x_t).
	\end{eqnarray*}
	We let
	\begin{eqnarray*}
		\pf_{mk,*}^{l+k}
		=
		\max
		\left\{
			\left| \pf^{l+k}_{mk,x} \right|:\, x\in\mathbb{T}_{l+k}
		\right\},
		\qquad
		A_{k,l}
		=
		\frac{\pf_{l+(m+1)k}}{ \pf_{mk,*}^{l+k} }.
	\end{eqnarray*}
	This way,
	\begin{eqnarray*}
		\ln \left|\pf_{l+(m+1)k} \right|
		=
		\ln \left|A_{k,l} \right|
		+
		\ln \left|\pf_{mk,*}^{l+k} \right|.
	\end{eqnarray*}
	Now, by hypothesis,
	\begin{eqnarray*}
		\p\left[
			\ln \left|\pf_{mk,*}^{l+k} \right| < mk(\free - c)
		\right]
		\leq
		p^{b^{l+k}},
	\end{eqnarray*}
	for some $p\in(0,1)$.
	Hence, by Borel-Cantelli, there exists an almost surely finite $k_0=k_0(\xi)\geq 1$ such that
	\begin{eqnarray*}
		\ln \left|\pf_{mk,*}^{l+k} \right| \geq mk(\free - c),
		\qquad
		\forall \, k \geq k_0.
	\end{eqnarray*}
	Hence, for $k\geq k_0$, it holds that
	\begin{eqnarray*}
		\frac{1}{l+(m+1)k}
		\ln |\pf_{l+(m+1)k}|
		&\geq&
		\frac{1}{l+(m+1)k}
		\ln |A_{k,l}|
		+
		\frac{mk}{l+(m+1)k}(\free-c)
		\\
		&=&
		\free-c
		+
		\frac{1}{l+(m+1)k}
		\ln |A_{k,l}|
		-
		\frac{l+k}{l+(m+1)k}(\free-c)
		\\
		&\geq&
		\free-2c
		+
		\frac{1}{l+(m+1)k}
		\ln |A_{k,l}|,
	\end{eqnarray*}
	for $k,m\geq 1$ large enough (for instance, if $m > \lfloor \frac{2+c}{c}\rfloor$ and $k\geq m$). Hence,
	\begin{eqnarray*}
		\liminf_{k \to \infty}
		\frac{1}{l+(m+1)k}
		\ln \left|\pf_{l+(m+1)k} \right|
		&\geq&
		\free-2c
		+
		\liminf_k	
		\frac{1}{l+(m+1)k}
		\ln |A_{k,l}|.
	\end{eqnarray*}
	Now, let $\kappa>0$. Let $x_*\in\mathbb{T}_{l+k}$ be such that $ \left| \pf_{mk,x_*}^{l+k} \right| = \left| \pf_{mk,*}^{l+k} \right|$. Then,
	\begin{eqnarray}
	\label{eq:estimate-A}
		\p\left[
			\left| A_{k,l} \right|<\kappa^{\tau(k+l)}
		\right]
		=
		\sum_{x\in\mathbb{T}_{l+k}}
		\p
		\left[
			\left.
			\left| A_{k,l} \right|<\kappa^{\tau(k+l)} 
			\right|
			\ x_*=x
		\right]
		\p[x_* = x].
	\end{eqnarray}
	Now,
	\begin{align*}
		\p
		\left[
		\left.
			|A_{k,l}|<\kappa^{l+k}
		\right| 
		\ x_*=x
		\right]
		&=
		\p\left[
			\xi_{l+k}(x)
			\in 
			B \left( z(\xi);\kappa^{l+k} \right)
		\right]
		\\
		&=
		\p
		\left[
			\xi(x) \cdot \prod^{l+k-1}_{t=1}\xi(x_t)
			\in 
			B \left( z(\xi);\kappa^{l+k} \right)
		\right],
	\end{align*}
	where $z(\xi)$ is independent of $\xi(x)$. Hence, with $\displaystyle \tilde{\xi}(x)=\prod^{l+k-1}_{t=1}\xi(x_t)$,
	\begin{align*}
		\p
		\left[
		\left.
			|A_{k,l}|<\kappa^{l+k}
		\right| \ x_*=x
		\right]
		&=
		\esp\left[
			\p\left[
				\xi(x) \in B\left( \tilde{\xi}(x)^{-1}z(\xi); \left| \tilde{\xi}(x) \right|^{-1}\kappa^{l+k}\right)
			\Big{|}
				\xi(x_1),\ldots,\xi(x_{l+k-1})
			\right]
		\right]
		\\
		&\leq
		C
		\esp\left[
			\left| \tilde{\xi}(x) \right|^{-\tau}\kappa^{\tau(l+k)}
		\right]
		\leq
		C \esp[|\xi|^{-\tau}]^{l+k-1} \cdot \kappa^{\tau(l+k)}.
	\end{align*}
	We get the same bound for \eqref{eq:estimate-A} by summing over $x$.
	Hence, if $\kappa$ is small enough, there exists an almost surely finite value $k_1=k_1(\xi)\geq 1$ such that
	\begin{eqnarray*}
		|A_{k,l}| \geq \kappa^{l+k},
		\qquad
		\forall \, k\geq k_1.
	\end{eqnarray*}
	As a consequence,
	\begin{eqnarray*}
		\liminf_{k \to \infty} 
		\frac{1}{l+(m+1)k}
		\ln |A_{k,l}|
		\geq
		\liminf_{k \to \infty}  \frac{l+k}{l+(m+1)k} \ln \kappa
		=
		\frac{1}{m+1}\ln \kappa.
	\end{eqnarray*}
	Hence, if $m$ is large enough,
	\begin{align*}
		\liminf_{k \to \infty} 
		\frac{1}{l+(m+1)k}
		\ln \left| \pf_{l+(m+1)k} \right|
		&\geq
		\free-3c. \qedhere
	\end{align*}
\end{proof}

\subsection{A generic computation}
\label{sec:generic}
	
Let us see how the arguments from the previous sections will be applied. The following scheme will be used in the proofs of Lemma \ref{thm:convergence-R3} for the region $\mathcal{R}3$ and Lemma \ref{thm:ending-r2} for the region $\mathcal{R}2$, where details will be omitted.

We consider the usual partition function $\pf_n(\xi)$ and a quantity of the form
\begin{eqnarray*}
	W_n(\xi)
	=
	\esp \left[ \left. \left| \pf_n(\xi) \right|^2 \right|  \mathcal{G} \right],
\end{eqnarray*}
where $\mathcal{G} \subset \mathcal{F}$ is some $\sigma$-algebra.
Assume that
\begin{eqnarray}
\label{eq:moments-bound-average}
	\esp
	\left[
	\left. 
			\frac{|\pf_n(\xi)|^{2\eta}}{|W_n(\xi)|^{\eta}}
			\right| 
			\mathcal{G}
	\right]
	\leq
	C,
\end{eqnarray}
almost surely, for some $\eta>1$, some finite constant $C>0$ and all $n\geq 1$. 
\begin{proposition}
\label{thm:joint-convergence}
	If the limit
	\begin{eqnarray}
		\lim_{n\to\infty} \frac{1}{2n} \ln W_n(\xi)
		=
		\free
\end{eqnarray}
holds $\p$-almost surely, then the limit
\begin{eqnarray*}
		\lim_{n\to\infty} \frac{1}{n} \ln |\pf_n(\xi)|
		=
		\free
\end{eqnarray*}
holds in probability.
Furthermore, under the $\tau$-condition, this limit holds almost surely.
\end{proposition}
\begin{proof}
We remove $\xi$ from the notation.
Integrating \eqref{eq:moments-bound-average} and applying Chebyshev's inequality, we obtain
\begin{align*}
	\p\left[
		\left| \pf_n \right|^2 \geq nW_n
	\right]
	\leq
	\frac{C}{n^{\eta}},
\end{align*}
for some finite $C>0$,
so that, by Borel-Cantelli, there exists an almost surely finite quantity $C(\omega)>0$ such that
\begin{align*}
	|\pf_n|^2 \leq C(\omega)nW_n,
\end{align*}
for all $n\geq 1$. Hence,
\begin{eqnarray}
\label{eq:estimate-pf-average}
	\limsup_{n\to\infty} \frac{1}{n} \ln |\pf_n|^2
	\leq
	\lim_{n\to\infty} \frac{1}{n} \ln W_n
	=
	\free,
	\quad\quad
	\p-\text{almost surely.}
\end{eqnarray}
To obtain a lower bound, let $c>0$. By Paley-Zygmund's inequality (Lemma \ref{thm:Paley-Zygmund} in Appendix \ref{app:PZ}), \eqref{eq:moments-bound-average} also implies that there exists $p\in(0,1)$ such that
\begin{eqnarray}
\label{eq:PZ}
	\p
	\left[
		\left| \pf_n \right|^2 \geq e^{-cn} W_n		
	\right]
	\geq
	1-p
	>
	0,
\end{eqnarray}
for all $n\geq 1$.  This entails that
\begin{eqnarray*}
	\p\left[
		\ln \left| \pf_n \right|^2 < \ln W_n -cn	
	\right]
	\leq p
	<1,
\end{eqnarray*}
for all $n\geq 1$. Hence, if $p<p'<1$,
\begin{align*}
	\p
	\left[
		\ln \left| \pf_n \right|^2 < n\free -2cn	
	\right]
	&=
	\p\left[
		\ln \left| \pf_n \right|^2 < n\free -2cn	, \, \ln W_n \geq n(\free-c)
	\right]
	\\
	&\quad+
	\p\left[
		\ln \left| \pf_n \right|^2 < n\free -2cn	, \, \ln W_n< n(\free-c)
	\right]
	\\
	&\leq
	\p\left[
		\ln \left| \pf_n \right|^2 <  \ln W_n -cn
	\right]
	+
	\p\left[
		\ln W_n < n(\free-c)
	\right]
	\leq
	p',
\end{align*}
for $n$ large enough. By Corollary \ref{thm:lower-tail}, we can make the left hand side as small as we want by taking $n$ large enough. This yields convergence in probability. Under the $\tau$-condition, we can apply Lemma \ref{thm:as-lower-bound} to obtain the almost sure lower bound. \qedhere
\end{proof}
 
\section{The region $\mathcal{R}1$.}
\label{sec:R1}
We assume that there exists $\theta \in (1,2]$ such that $G(\theta)<\ln(b|\esp[\xi]|)$.
In particular, $\esp[\xi]\neq 0$.  In this case, the sequence $(\mart_n(\xi))_{n \ge 1}$ is a uniformly integrable martingale and is then convergent in $\mathbb{L}^1$ and almost surely. Let
\begin{eqnarray*}
	\mart_{\infty}(\xi)
	=
	\lim_{n\to\infty} \mart_n(\xi).
\end{eqnarray*}
The following lemma will allow us to compute the free energy in this regime.
\begin{lemma}
	\label{thm:zero-one-law}
	Assume that there exists $\theta \in (1,2]$ such that $G(\theta)<\ln(b|\esp[\xi]|)$ and that the law of $\xi$ is continuous. Then,
	\begin{eqnarray*}
		\p\left[
			\mart_{\infty}(\xi)=0
		\right]
		=
		0 \,\, \text{or} \,\, 1.
	\end{eqnarray*}
\end{lemma}
We defer the proof to the end of the section and compute the free energy in the region $\mathcal{R}1$.

\begin{lemma}
	Under the hypothesis of Theorem \ref{thm:main}, it holds that
	\begin{eqnarray*}
		\lim_{n\to\infty}
		\frac{1}{n}\ln |\pf_n(\xi)|
		=
		\ln\left(
			b |\esp[\xi]|
		\right),
		\qquad
		\p-\text{almost surely}.
	\end{eqnarray*}
\end{lemma}
\begin{proof}
As $(\mart_n(\xi))_n$ is uniformly integrable, it holds that
\begin{eqnarray*}
	\esp[\mart_{\infty}(\xi)]
	=
	\lim_{n\to\infty}\esp[\mart_{n}(\xi)]
	=
	1,
\end{eqnarray*}
so that
\begin{eqnarray*}
		\p\left[
			\mart_{\infty}(\xi)=0
		\right]
		=
		0.
\end{eqnarray*}
Hence,
\begin{eqnarray*}
	\lim_{n\to\infty}
	\frac{1}{n}\ln \mart_n(\xi)
	=
	\lim_{n\to\infty}
	\frac{1}{n}\ln \mart_{\infty}(\xi)
	=
	0,
\end{eqnarray*}
almost surely, which shows that
\begin{equation*}
	\lim_{n\to\infty}
	\frac{1}{n}\ln |\pf_n(\xi)|
	=
	\lim_{n\to\infty}
	\frac{1}{n}\ln | \esp[\pf_n(\xi)]|
	=
	\ln\left(
		b |\esp[\xi]|
	\right). \qedhere
\end{equation*}
\end{proof}

The following simple lemma is the key to show Lemma \ref{thm:zero-one-law}.
\begin{lemma}
	Let $z_1,\ldots,z_m \in \C$ and let $\omega_1,\ldots,\omega_m$ be complex-valued independent random variables with continuous laws. Then,
	\begin{eqnarray*}
		\p\left[
			\sum^m_{k=1} z_k \omega_k=0
		\right]
		=
		0,
	\end{eqnarray*}
	if and only if $z_1=\cdots=z_m=0$.
\end{lemma}
\begin{proof}
	This follows from the fact that, if $(z_1,\cdots,z_m)\neq 0$, the subspace of $\C^m$ defined by
	\begin{eqnarray*}
		\left\{
			( \omega_1,\ldots,\omega_m )\in\C^m: \, \sum^m_{k=1} z_k \omega_k=0
		\right\}
	\end{eqnarray*}
	has zero Lebesgue measure.
\end{proof}

\begin{proof}[Proof of Lemma \ref{thm:zero-one-law}]
	We have the simple identity
	\begin{eqnarray*}
		\pf_{n+1}(\xi)
		=
		\sum_{\walk\in\walks_{n+1}}
			\xi(\walk_1) \prod^{n+1}_{t=2}\xi(\walk_t)
		=
		\sum_{x\in\T_1} \xi(x) \pf_{n,x}(\xi),
	\end{eqnarray*}
	so that
	\begin{eqnarray*}
		\mart_{n+1}(\xi)
		=
		b^{-1}\sum_{x\in\T_1} \frac{\xi(x)}{\esp[\xi]} \mart_{n,x}(\xi),
	\end{eqnarray*}
	with a, hopefully, self-explaining notation.
	Now, all these martingales are almost surely convergent and we get that
	\begin{eqnarray*}
		\mart_{\infty}(\xi)
		=
		b^{-1}\sum_{x\in\T_1} \frac{\xi(x)}{\esp[\xi]} \mart_{\infty,x}(\xi),
	\end{eqnarray*}
	where the collection $\{\mart_{\infty,x}(\xi):\, x\in\T_1\}$ is i.i.d. with common law agreeing with that of $\mart_{\infty}(\xi)$. Applying the last lemma to the random variables $\left\{ \xi(x)/\esp[\xi]: \ x\in\T_1 \right\}$,
	\begin{eqnarray*}
		\p\left[
			\mart_{\infty}(\xi)=0
		\right]
		=
		\p\left[
			\mart_{\infty,x}(\xi)=0,\, \forall\, x\in\T_1
		\right]
		=
		\p\left[
			\mart_{\infty}(\xi)=0
		\right]^b.
	\end{eqnarray*}
	This finishes the proof.
\end{proof}



\section{The region $\mathcal{R}2$}
\label{sec:R2}

This region will be divided in two parts. The identification of the free energy in each part is given in Lemma \ref{thm:R2-easy-part} and Lemma \ref{thm:ending-r2} respectively.
\subsection{The region $\amin\leq 1$}

\begin{lemma}
\label{thm:R2-easy-part}
	Under the hypothesis of Theorem \ref{thm:main}, if $\amin<1$, then
	\begin{eqnarray*}
		\lim_{n\to\infty}
		\frac{1}{n}\ln |\pf_n(\xi)|
		=
		G(\amin),
		\qquad
		\p-\text{almost surely}.
	\end{eqnarray*}	
\end{lemma}
\begin{proof}
Recall that, in this region, it holds that $\free ( \left| \xi \right|^2 ) = 2 \free(|\xi|) = 2 G(\amin)$.
By the triangle inequality,
\begin{eqnarray*}
	|\pf_n(\xi)| \leq \pf_n(|\xi|),
\end{eqnarray*}
so that
\begin{eqnarray}
\label{eq:sd-upper-bound}
	\limsup_{ n \to \infty}
	\frac{1}{n}\ln |\pf_n(\xi)|
	\leq
	\limsup_{ n \to \infty }
	\frac{1}{n}\ln \pf_n(|\xi|)
	=
	G(\amin).
\end{eqnarray}
Next, let
\begin{eqnarray*}
	E_n(\xi)
	=
	\sum_{\walk\neq\walk'\in\walks_n} 
	\prod^n_{t=1} \xi(\walk_t)\overline{\xi(\walk'_t)}.
\end{eqnarray*}
Then,
\begin{eqnarray}
	\nonumber
	|\pf_n(\xi)|^2
	&=&
	\pf_n ( \left| \xi \right|^2 ) +E_n(\xi)
	=
	\pf_n ( \left| \xi \right|^2 )
	\left(
		1+\frac{E_n(\xi)}{\pf_n ( \left| \xi \right|^2)}
	\right)
	\\
	&\geq&
	\label{eq:factorization-lower-bound}
	\pf_n ( \left| \xi \right|^2 )
	\left(
		1+\frac{E_n(\xi)}{\pf_n ( \left| \xi \right| )^2}
	\right).
\end{eqnarray}
Suppose that
\begin{eqnarray*}
	\liminf_{ n \to \infty} 
	\frac{1}{n}\ln
	\left(
		1+\frac{E_n(\xi)}{\pf_n ( \left| \xi \right| )^2}
	\right)
	<
	0.
\end{eqnarray*}
Then, there exist $c_1,c_2>0$ and a subsequence indexed by $(n_k)_{k \ge 1}$ such that
\begin{eqnarray*}
	1+\frac{E_{n_k}(\xi)}{\pf_{n_k}(|\xi|)^2}
	\leq
	c_1 e^{-c_2 n_k},
\end{eqnarray*}
for all $k\geq 1$.
Hence,
\begin{eqnarray*}
	E_{n_k}(\xi)
	\leq 
	\left( c_1 e^{-c_2 n_k}-1 \right)
	\pf_{n_k}(|\xi|)^2,
\end{eqnarray*}
which is negative for $k$ large enough, so that
\begin{eqnarray*}
	\left( 1-c_1 e^{-c_2 n_k} \right)
	\pf_{n_k}(|\xi|)^2
	\leq 
	\left|
		E_{n_k}(\xi)
	\right|
	\leq
	E_{n_k}(|\xi|),
\end{eqnarray*}
by the triangle inequality.
Hence,
\begin{eqnarray*}
	\pf_{n_k} ( \left| \xi \right|^2 )
	+
	\left( 1-c_1 e^{-c_2 n_k} \right)
	\pf_{n_k}(|\xi|)^2
	\leq
	\pf_n ( \left| \xi \right|^2 )
	+
	E_{n_k}(|\xi|)
	=
	\pf_{n_k}(|\xi|)^2,
\end{eqnarray*}
so that
\begin{eqnarray*}
	\pf_{n_k} ( \left| \xi \right|^2 )
	\leq
	c_1e^{-c_2n_k}
	\pf_{n_k}(|\xi|)^2,
\end{eqnarray*}
and, after computing the free energies,
\begin{eqnarray*}
	\free ( \left| \xi \right|^2 )
	\leq
	-c_2+2\free(|\xi|)<2\free(|\xi|),
\end{eqnarray*}
which contradicts the fact that $\free ( \left| \xi \right|^2 )=2\free(|\xi|)=2G(\amin)$.
Hence,
\begin{eqnarray*}
	\liminf_{ n \to \infty} 
	\frac{1}{n}\ln
	\left(
		1+\frac{E_n(\xi)}{\pf_n(|\xi|)^2}
	\right)
	\geq
	0,
\end{eqnarray*}
and, going back to \eqref{eq:factorization-lower-bound}, we conclude that
\begin{eqnarray*}
	\liminf_{n \to \infty} 
	\frac{1}{n}\ln |\pf_n(\xi)|^2
	\geq
	\liminf_{ n \to \infty} 
	\frac{1}{n}\ln \pf_n ( \left| \xi \right|^2 )
	=
	2G(\amin).
\end{eqnarray*}
Together with \eqref{eq:sd-upper-bound}, we obtain the \(\p\)-almost surely limit
\begin{equation*}
	\lim_{n\to\infty} 
	\frac{1}{n}\ln |\pf_n(\xi)|^2
	=
	2G(\amin). \qedhere
\end{equation*}
\end{proof}

\subsection{The region $1<\amin \leq 2$: averaging the phases}
We write $\xi=e^{\omega+\iota\theta}$ and we suppose that the family of random variables $\omega$ and $\theta$ are independent. First,
\begin{eqnarray*}
	|\pf_n(\xi)|^2
	=
	\sum_{\walk,\walk' \in \walks_n}
		\prod^n_{t=1} \xi(\walk_t)\overline{\xi \left( \walk'_t \right)}
	=
	\sum_{\walk,\walk' \in \walks_n}
		\prod^n_{t=1}e^{\iota \left[ \theta(\walk_t)-\theta \left( \walk'_t \right) \right]}
		\prod^n_{t=1} \left| \xi(\walk_t) \right| \, | \xi \left( \walk'_t \right) |.
\end{eqnarray*}
We will adopt the notation \( \esp \left[ \left. \cdot \right| \omega \right] := \esp \left[ \left. \cdot \right| \sigma \left( \omega (x): \ x \in \mathbb{T} \right) \right] \). Denoting by $\walk \wedge \walk'$ the first time the paths $\walk$ and $\walk'$ separate,
\begin{eqnarray*}
	\esp\left[
		\left. \left| \pf_n(\xi) \right|^2
		\right| \omega
	\right]
	&=&
	\esp
	\left[
	\left. 
	\sum_{\walk,\walk' \in \walks_n}
		\prod^n_{t=1} \xi(\walk_t)\overline{\xi(\walk'_t)}
		\right|  \omega
	\right]
	=
	\esp
	\left[
	\left. 
	\sum_{\walk,\walk' \in \walks_n}
		\prod^n_{t=1}e^{\iota[\theta(\walk_t)-\theta(\walk'_t)]}
		\prod^n_{t=1} |\xi(\walk_t)| | \xi \left( \walk'_t \right) |\,
	\right|  \omega
	\right]
	\\
	&=&
	\sum_{\walk,\walk' \in \walks_n}
		\prod^n_{t=1}
			\esp\left[
				e^{\iota[\theta(\walk_t)-\theta(\walk'_t)]}
			\right]
		\prod^n_{t=1} |\xi(\walk_t)| \, | \xi \left( \walk'_t \right) |
	\\
	&=&
	\sum_{\walk\in\walks_n}
		\prod^n_{t=1}|\xi(\walk_t)|^2
	+
	\sum_{\walk\neq\walk' \in \walks_n}
		\prod_{t>\walk \wedge \walk'}	
			\left|
				\esp\left[
					e^{\iota \theta}
				\right]
			\right|^2	
		\prod^n_{t=1} |\xi(\walk_t)| \, | \xi \left( \walk'_t \right) |
	\\
	&=:&
	\pf_n ( \left| \xi \right|^2 )
	+
	\eterm_n(\omega).
\end{eqnarray*}
Note that $\eterm_n(\omega)\geq 0$ so that
\begin{eqnarray*}
	\esp
	\left[
	\left.
		\left| \pf_n(\xi) \right|^2
	\right|
	\omega
	\right]
	&\geq&
	\pf_n ( \left| \xi \right|^2 ),
\end{eqnarray*}
and, according to the discussion in Section \ref{sec:comparison},
\begin{eqnarray}
\label{eq:free-energy-conditionned-lower-bound}
	\liminf_{n\to\infty}
	\frac{1}{n}\ln \esp
	\left[
	\left.
	\left| \pf_n(\xi) \right|^2
	\right|
	\omega
	\right]
	\geq
	\lim_{n\to\infty}
	\frac{1}{n} \ln \pf_n ( \left| \xi \right|^2 )
	=
	2G(\amin).
\end{eqnarray}
The following corresponds to Theorem 4.2 and Theorem 4.1 (a) in \cite{DES-Complex} and is the key estimate to study this regime. We reproduce the proof for the convenience of the reader.
\begin{lemma}
\label{thm:key-independent-phase-and-radii}
Assume that $\omega$ and $\theta$ are independent. Then,
\begin{eqnarray*}
	\frac{
	\esp\left[ \left. \left| \pf_n(\xi) \right|^4
	\right|
	\omega
	\right]
	}{\esp \left[ \left. \left| \pf_n(\xi) \right|^2
	\right|
	\omega
	\right]^2}
	\leq 3,
	\qquad
	\forall n\geq 1.
\end{eqnarray*}
Furthermore, in the region $\mathcal{R}1$, there exists $\eta>1$ such that
\begin{align*}
	\sup_{n \ge 1}
	\frac{\esp\left[ \esp \left[ \left. \left| \pf_n(\xi) \right|^2
		\right|
		\omega
		\right]^{\frac{\eta}{2}}\right]}{ \left|\esp \left[ \pf_n(\xi) \right] \right|^{\eta}}
	<
	\infty.
\end{align*}
\end{lemma}
\begin{proof}
	To simplify the notation, we let
	\begin{align*}
		\xi_{\walk}
		=
		\xi_n(\walk)=\prod^n_{t=1}\xi(\walk_t),
		\qquad
		\Gamma_{\walk}
		=
		\Gamma_n(\walk)=\frac{\xi_n(\walk)}{|\xi_n(\walk)|},
	\end{align*}
	for $\walk\in\walks_n$.
	Then,
	\begin{align*}
		\esp
		\left[
		\left.
		\left| \pf_n(\xi) \right|^2
		\right|
		\omega
		\right]
		&=
		\sum_{\walk^1,\walk^2\in\walks_n}
			\left| \xi_{\walk^1}\xi_{\walk^2} \right|
			\esp \left[ \Gamma_{\walk^1}\overline{\Gamma_{\walk^2}} \right],
		\\
		\esp\left[
		\left.
		\left| \pf_n(\xi) \right|^4
		\right|
		\omega
		\right]
		&=
		\sum_{\walk^1,\ldots,\walk^4\in\walks_n}
			\left| \xi_{\walk^1}\xi_{\walk^2}\xi_{\walk^3}\xi_{\walk^4}
			\right|
			\esp
			\left[
			\Gamma_{\walk^1}\overline{\Gamma_{\walk^2}}\Gamma_{\walk^3}\overline{\Gamma_{\walk^4}}
			\right].
	\end{align*}
	Squaring the first identity,
	\begin{align*}
		\esp
		\left[
		\left.
		\left| \pf_n(\xi) \right|^2
		\right|
		\omega
		\right]^2
		&=
		\sum_{\walk^1,\ldots,\walk^4\in\walks_n}
			\left|
			\xi_{\walk^1}\xi_{\walk^2}\xi_{\walk^3}\xi_{\walk^4}
			\right|
			\esp
			\left[
			\Gamma_{\walk^1}\overline{\Gamma_{\walk^2}}
			\right]
			\esp
			\left[
			\Gamma_{\walk^3}\overline{\Gamma_{\walk^4}}
			\right].
	\end{align*}
	Note that the product of the amplitudes remains unchanged by permutations of the four paths so that, denoting the group of permutations over 4 objects by $P_4$,
	\begin{align}
	\nonumber
			\esp \left[
			\left.
			\left| \pf_n(\xi) \right|^4
			\right|
			\omega
			\right]
		&=
		\frac{1}{24}
		\sum_{\walk^1,\ldots,\walk^4\in\walks_n}
			\prod^4_{j=1}
			\left| \xi_{s^j} \right|
			\sum_{\sigma\in P_4}			
			\esp
			\left[
				\Gamma_{\walk^{\sigma(1)}}\overline{\Gamma_{\walk^{\sigma(2)}}}\Gamma_{\walk^{\sigma(3)}}\overline{\Gamma_{\walk^{\sigma(4)}}}
			\right],
		\\
		\label{eq:two-walks}
		\esp
		\left[
		\left.
		\left| \pf_n(\xi) \right|^2
		\right|
		\omega
		\right]^2
		&=
		\frac{1}{24}
		\sum_{\walk^1,\ldots,\walk^4\in\walks_n}
			\prod^4_{j=1}
			\left| \xi_{s^j} \right|
			\sum_{\sigma\in P_4}		
				\esp[\Gamma_{\walk^{\sigma(1)}}\overline{\Gamma_{\walk^{\sigma(2)}}}]
				\esp[\Gamma_{\walk^{\sigma(3)}}\overline{\Gamma_{\walk^{\sigma(4)}}}].
	\end{align}
	Now, note that, for each choice of four paths $\walk^i,\, i=1,\ldots,4$, we can find a permutation $\sigma\in P_4$ such that $\walk^{\sigma(i)}$ is the $i$-th path that separates from the others (with some arbitrary convention if two paths separate at the same time). In this case,
	\begin{align*}
		\esp
		\left[
			\Gamma_{\walk^{1}}
			\overline{\Gamma_{\walk^2}}
			\Gamma_{\walk^{3}}
			\overline{\Gamma_{\walk^{4}}}
		\right]
		\leq
		\esp
		\left[
			\Gamma_{\walk^{\sigma(1)}}\overline{\Gamma_{\walk^{\sigma(2)}}}\Gamma_{\walk^{\sigma(3)}}\overline{\Gamma_{\walk^{\sigma(4)}}}
		\right]
		&=
		\esp
		\left[
		\Gamma_{\walk^{\sigma(1)}}
		\overline{\Gamma_{\walk^{\sigma(2)}}}
		\right]
		\esp
		\left[
		\Gamma_{\walk^{\sigma(3)}}
		\overline{\Gamma_{\walk^{\sigma(4)}}}
		\right]
		=
		\left| \esp[\xi] \right|^{n \left( \walk^1,\walk^2,\walk^3,\walk^4 \right)},
	\end{align*}
	where $n\left(\walk^1,\walk^2,\walk^3,\walk^4 \right)$ is the number of sites visited by an odd number of walks. 
	Hence,
	\begin{align}
	\label{eq:conditioning-omega-upper-bound}
		\esp
		\left[
		\left.
		\left| \pf_n(\xi) \right|^4
		\right|
		\omega
		\right]
		&\leq 
		\sum_{\walk^1,\ldots,\walk^4\in\walks_n}
			\Big{(}
				\prod^4_{j=1} \left| \xi_{s^j} \right|
			\Big{)}
			\left| \esp[\xi] \right|^{n \left( \walk^1,\walk^2,\walk^3,\walk^4 \right)}.
	\end{align}
	On the other hand, given four paths $\walk^i,\, i=1,\ldots,4$, there are at least eight permutations such that
	\begin{align*}
		\esp
		\left[
		\Gamma_{\walk^{\sigma(1)}}\overline{\Gamma_{\walk^{\sigma(2)}}}\Gamma_{\walk^{\sigma(3)}}\overline{\Gamma_{\walk^{\sigma(4)}}}
		\right]
		&=
		\esp
		\left[
		\Gamma_{\walk^{\sigma(1)}}
		\overline{\Gamma_{\walk^{\sigma(2)}}}
		\right]
		\esp
		\left[
		\Gamma_{\walk^{\sigma(3)}}
		\overline{\Gamma_{\walk^{\sigma(4)}}}
		\right]
		=
		\left| \esp[\xi] \right|^{n \left( \walk^1,\walk^2,\walk^3,\walk^4 \right)},
	\end{align*}
	Together with the positivity of the expected values in  \eqref{eq:two-walks},
		\begin{align*}
			\esp
			\left[
			\left.
			\left| \pf_n(\xi) \right|^2
			\right|
			\omega
			\right]^2
			&\geq
			\frac{8}{24}
			\sum_{\walk^1,\ldots,\walk^4\in\walks_n}
			\Big{(}
				\prod^4_{j=1}|\xi_{s^j}|
			\Big{)}
			\left| \esp[\xi] \right|^{n \left( \walk^1,\walk^2,\walk^3,\walk^4 \right)}.
	\end{align*}
	Together with \eqref{eq:conditioning-omega-upper-bound}, this proves the first estimate. The second one follows by straightforward adaptations of the argument leading to the first estimate in Lemma \ref{thm:moment-estimates}, replacing the conditional probability given $\calf_{n-1}$ in \eqref{eq:aux-moments} by conditioning with respect to $\omega$.
\end{proof}

The next lemma will allow us to identify the free energy in this regime.
\begin{lemma}
Assume that $1\leq \amin < 2$, $G(\amin)>\ln(b\esp[\xi]|)$ and, in addition to the hypotheses of Theorem \ref{thm:main}, assume that the random variables $\omega$ and $\theta$ are independent. Then,
\begin{eqnarray*}
	\lim_{n\to\infty}
	\frac{1}{n}\ln \esp \left[ \left. \left| \pf_n(\xi) \right|^2 \right| \omega \right]
	=
	2G(\amin),
	\qquad
	\p-\text{almost surely}.
\end{eqnarray*}
\end{lemma}
\begin{proof}
	We already obtained the lower bound in \eqref{eq:free-energy-conditionned-lower-bound}.
	The upper bound requires more work. 
	
	Assume that $1\leq\amin\leq 2$ and write $\xi = e^{\iota \theta} |\xi|=e^{\iota \theta + \omega}$. We will define a new environment $e^{\iota \tilde{\theta}}|\xi|$ such that $\gamma \mapsto \left| \esp \left[e^{\iota \gamma \tilde{\theta}} \right] \right|$ is decreasing in $[0,1]$. For this, we let $t= \left| \esp \left[ e^{\iota \theta} \right] \right|$, $z=t+\iota \sqrt{1-t^2}$ and
	\begin{eqnarray*}
		e^{\iota \tilde{\theta}}
		=
		\begin{cases}
			z, & \text{with probability} \ \frac{1}{2},\\
			\overline{z}, & \text{with probability} \ \frac{1}{2}.
		\end{cases}
	\end{eqnarray*}
	Note that 
	\begin{eqnarray*}
		\esp[|\pf_n(\xi)|^2|\, \omega]
		=
		\esp[|\pf_n(e^{\iota \tilde{\theta}} |\xi|)|^2|\, \omega].
	\end{eqnarray*}
	Now, there exists $0<\gamma_0<1$ such that $\pf_n \left( e^{\iota \gamma \tilde{\theta}} |\xi| \right)$ will be in the region $\mathcal{R}1$ for all $0\leq \gamma < \gamma_0$ and on the $\mathcal{R}1-\mathcal{R}2$ boundary for $\gamma=\gamma_0$. 
	Hence, for $0\leq \gamma < \gamma_0$, the second estimate in Lemma \ref{thm:key-independent-phase-and-radii} and a Borel-Cantelli argument imply
	\begin{align*}
		\limsup_{n\to\infty} \frac{1}{n} 
		\ln \esp
		\left[ \left. \left| \pf_n \left(e^{\iota \gamma \tilde{\theta}} |\xi| \right) \right|^2 \right| \omega \right]
		&\leq
		\lim_{n\to\infty}\frac{1}{n} \ln 
		\left| \esp \left[\pf_n \left(e^{\iota \gamma \tilde{\theta}} |\xi| \right)\right] \right|\\
		&=
		\lim_{n\to\infty}\frac{1}{n} \ln \left| \pf_n \left(e^{\iota \gamma \tilde{\theta}} |\xi| \right) \right|
		=:
		\free(\gamma),
	\end{align*}
	almost surely. Let $\delta>0$. As the model falls into the region $\mathcal{R}2$ when $\gamma>\gamma_0$ and $\lim_{\gamma\to\gamma_0^-}f(\gamma)=2G(\amin)$, there exists $0<\gamma_1<\gamma_0$ such that $\free(\gamma_1)<2G(\amin)+\delta$. Now, by construction, the function
	\begin{eqnarray*}
		[0,1] \ni \gamma \mapsto \esp
		\left[ \left. \left| \pf_n \left(e^{\iota \gamma \tilde{\theta}} |\xi| \right) \right|^2 \right| \omega \right],
	\end{eqnarray*}
	is decreasing. Hence,
	\begin{align*}
		\limsup_{n\to\infty}
		\frac{1}{n} \ln \esp \left[ \left. \left| \pf_n(\xi) \right|^2 \right| \omega \right]
		&=
		\limsup_{n\to\infty} \frac{1}{n} 
		\ln \esp \left[ \left. \left| \pf_n \left(e^{\iota \tilde{\theta}} |\xi| \right) \right|^2 \right| \omega \right]
		\\
		&\leq
		\limsup_{n\to\infty} \frac{1}{n} 
		\ln \esp
		\left[ \left. \left| \pf_n \left(e^{\iota \gamma_1 \tilde{\theta}} |\xi| \right) \right|^2 \right| \omega \right]
		=
		\free(\gamma_1)
		<2G(\amin)+\delta,
	\end{align*}
	almost surely. This finishes the proof. \qedhere
\end{proof}
\begin{lemma}
\label{thm:ending-r2}
Assume that $1\leq \amin < 2$, $G(\amin)>\ln(b\esp[\xi]|)$ and, in addition to the hypotheses of Theorem \ref{thm:main}, assume that the random variables $\omega$ and $\theta$ are independent. Then,
\begin{eqnarray*}
		\lim_{n\to\infty}
		\frac{1}{n}\ln |\pf_n(\xi)|
		=
		G(\amin),
	\end{eqnarray*}
	in probability. 
	Under the $\tau$-condition, the convergence holds $\p$-almost surely.
\end{lemma}
\begin{proof}
	Thanks to the previous lemma and the first estimate in Lemma \ref{thm:key-independent-phase-and-radii}, the result follows by an application of Proposition \ref{thm:joint-convergence} with $\mathcal{G}=\sigma(\omega(x):\, x\in\mathbb{T})$. 
\end{proof}

\section{The region $\mathcal{R}3$.}
\label{sec:R3}
The next lemma finishes the proof of Theorem \ref{thm:main}.
\begin{lemma}
\label{thm:convergence-R3}
	Under the hypothesis of Theorem \ref{thm:main}, it holds that
	\begin{eqnarray*}
		\lim_{n\to\infty} \frac{1}{n} \ln |\pf_n(\xi)|
		=
		\frac{1}{2}\ln 
		\left( b\esp \left[ \left| \xi \right|^2 \right]\right)
		=G(2),
	\end{eqnarray*}
	in probability, in the region $\mathcal{R}3$. Under the $\tau$-condition, the convergence holds $\p$-almost surely.
\end{lemma}
\begin{proof}
	Recall that
	\begin{eqnarray*}
		\martx_n(\xi)
		=
		\frac{|\pf_n(\xi)|^2}{\esp \left[ \left| \pf_n(\xi) \right|^2 \right]},
	\end{eqnarray*}
	and recall that there exists $\nu>1$ such that
	\begin{eqnarray*}
		\sup_{n \ge 1}
		\esp\left[ \martx_n(\xi)^{\nu} \right] < \infty.
	\end{eqnarray*}
	In Section \ref{sec:comparison}, we already noted that
	\begin{align*}
		\lim_{n\to\infty}\frac{1}{2n}\ln \esp \left[ \left| \pf_n(\xi) \right|^2 \right]
		=
		G(2).
	\end{align*}
	The result then follows by an application of Proposition \ref{thm:joint-convergence} with $\mathcal{G}=\{\emptyset,\Omega\}$.
\end{proof}

\section{Independent phases and radii}
\label{sec:independent}
In this section, we will analyse how the three regions described in Theorem \ref{thm:main} are characterized in the case of independent phases and radii. The discussion below contains the proof of Corollary \ref{cor:indep}.

Recall that
\begin{align*}
	\xi_{ \beta, \gamma } = e^{\beta \omega + \iota \gamma \theta},
	\qquad
	\lambda_\R (\beta)
	= \ln \esp \left[ e^{\beta \omega}  \right],
	\qquad
	\lambda_\C (\gamma) 
	= -\ln \left| \esp \left[ e^{\iota \gamma \theta} \right] \right|.
\end{align*}
Hence, the function $G$ is given by
\begin{equation}
	\label{derida g xi_br}
	G(\alpha)
	=
	G_\beta(\alpha)
	=
	\frac{1}{\alpha}
	\ln
	\left( b \esp \left[ \left| \xibg \right|^{\alpha} \right] \right)
	=
	\frac{\ln b + \lrad(\alpha \beta)}{\alpha},
\end{equation} 
and we have
\begin{equation}
	\label{derida g' xi_br}
	G'_\beta(\alpha)
	=
	\frac{ (\alpha \ \beta) \ \lambda_\R' (\alpha \ \beta) - \lambda_\R (\alpha \ \beta) - \ln b }{\alpha^2}. 
\end{equation} 
As a preliminary, we note that $\amin > 2$ is equivalent to $G_{\beta}'(2) < 0$, which is equivalent to
\begin{align*}
	2\beta \lambda_{\R}'(2\beta)-\lambda_{\R}(2\beta)< \ln b,
\end{align*}
i.e., $\beta< \beta_0$. Note that this corresponds to the $\mathbb{L}^2$-region for the model with $\gamma=0$.
In the same vein, the condition $\amin>1$ is equivalent to $G_{\beta}'(1)<0$ which can be rewritten as
\begin{align*}
	\beta\lrad'(\beta)-\lrad(\beta)<\ln b,
\end{align*}
which corresponds to the weak disorder region for the model with $\gamma=0$, i.e., $\beta<\beta_c$.

Finally, note that $\amin=\amin(\beta)$ satisfies
\begin{align}
	\label{gregorio 3x}
	\amin\beta\lrad'(\amin\beta)-\lrad(\amin\beta)=\ln b.
\end{align}
Hence, $\amin=\beta_c/\beta$. Furthermore,
\begin{align*}
	G_{\beta}(\amin)
	=
	\frac{\ln b + \lrad(\beta_c)}{\beta_c/\beta}
	=
	\frac{\beta_c\lrad'(\beta_c)}{\beta_c/\beta}
	=
	\beta\lrad'(\beta_c).
\end{align*}

The region $\mathcal{R}1$ is characterized by the condition that there exists $\alpha\in(1,2]$ such that $G_{\beta}(\alpha)<\ln(b|\esp[\xi_{\beta,\gamma}]|)$ i.e.
\begin{align}
	\label{eq:R1-indep}
	\frac{\ln b + \lrad(\alpha \beta)}{\alpha} < \ln b + \lrad(\beta)-\lph(\gamma)
	=
	G_{\beta}(1)-\lph(\gamma).
\end{align}
We split our discussion into three cases. If $1<\amin\leq 2$, then the above condition is equivalent to
\begin{align*}
	G_{\beta}(\amin) < \ln b + \lrad(\beta)-\lph(\gamma),
\end{align*}
i.e., \( \beta\lrad'(\beta_c)-\lrad(\beta)+\lph(\gamma)<\ln b. \)

We have seen above that $1<\amin\leq 2$ is equivalent to $\beta_0 \leq \beta < \beta_c$.

Next, assume that $\amin>2$ (which, in particular, implies that $\beta<\beta_0$). In this case, the function $\alpha\mapsto G_{\beta}(\alpha)$ is decreasing in the interval $(1,2]$ and Condition \eqref{eq:R1-indep} is therefore equivalent to $G_{\beta}(2)<G_{\beta}(1)-\lph(\gamma)$, i.e.,
\begin{align*}
	\frac{\ln b + \lrad(2\beta)}{2}<\ln b + \lrad(\beta) - \lph(\gamma),
\end{align*}
or, equivalently, \(\lrad(2\beta)-2\lrad(\beta)+2\lph(\gamma)<\ln b.\)

We are left with the possibility that $\amin\leq 1$, which, in particular, implies that $\beta \geq \beta_c$. In this case, the function $\alpha\mapsto G_{\beta}(\alpha)$ is increasing in the interval $(1,2]$ and $G_{\beta}(\alpha) \geq G(1)$ for all $(1,2]$. Hence, Condition \eqref{eq:R1-indep} cannot be satisfied.

Finally, note that, in the whole $\mathcal{R}1$ region, we have
\begin{align*}
	\free(\beta,\gamma) 
	= 
	\fone(\beta,\gamma)
	=
	\ln(b|\esp[\xi_{\beta,\gamma}]|)
	=
	\ln b + \lrad(\beta)-\lph(\gamma).
\end{align*}

The region $\mathcal{R}3$ is characterized by the condition that $\amin>2$ (so that $\beta>\beta_0$) and  $G_{\beta}(2)>\ln \left( b \left| \esp \left[\xi_{\beta,\gamma} \right] \right| \right)$, i.e.,
\begin{align*}
	G_{\beta}(2)
	=
	\frac{1}{2}\ln \left( b\esp \left[ \left| \xi_{\beta,\gamma} \right|^2 \right] \right)
	=
	\frac{1}{2}\left( \ln b + {\lrad(2\beta)} \right)
	>
	\ln b + \lrad(\beta)-\lph(\gamma).
\end{align*}
We can rewrite this as
\begin{align*}
	\lrad(2\beta)-2\lrad(\beta)+2\lph(\gamma)>\ln b.
\end{align*}
Furthermore,
\begin{align*}
	\free (\beta,\gamma)  = \fthree (\beta,\gamma)  = G_{\beta}(2) = \frac{1}{2}\left( \ln b + \lrad(2\beta)\right).
\end{align*}

Finally, the region $\mathcal{R}2$ has two parts. The first one is characterized by the condition $\amin<1$. We have seen above that this is equivalent to
\begin{align*}
	\beta\lrad'(\beta)-\lrad(\beta)>\ln b,
\end{align*}
i.e., $\beta>\beta_c$.

The second possibility is given by the conditions $1\leq \amin < 2$ (i.e., $\beta_0 \leq \beta < \beta_c$) and $G_{\beta}(\amin)>\ln(b|\esp[\xi_{\beta,\gamma}]|)$\, i.e.,
\begin{align*}
	\beta\lrad'(\beta_c)
	>
	\ln b + \lrad(\beta)-\lph(\gamma),
\end{align*}
which can rewrite as
\begin{align*}
	\beta\lrad'(\beta_c)-\lrad(\beta)+\lph(\gamma)>\ln b.
\end{align*}
In both cases,
\begin{align*}
	\free (\beta,\gamma) = \ftwo (\beta,\gamma) = G_{\beta}(\amin) = \beta \lrad'(\beta_c).
\end{align*}
The above discussion finishes the proof of the Corollary  \ref{cor:indep}.

\appendix

\section{The Paley-Zygmund inequality}
\label{app:PZ}
We provide the version of the Paley-Zigmund inequality used in Section \ref{sec:generic}.
\begin{lemma}
\label{thm:Paley-Zygmund}
Let $X$ be a non-negative random variable and suppose that there exists $\nu>1$ such that
\begin{align*}
	\frac{\esp[X^{\nu}]}{\esp[X]^{\nu}}\leq B,
\end{align*}
for some $B\in(0,\infty)$. Then,
\begin{align*}
	\p\left[
		X>\theta \esp[X]
	\right]
	\geq
	B^{-\frac{1}{\nu-1}}
	\left(
		1-\theta
	\right)^{\frac{\nu}{\nu-1}},
\end{align*}
for all $\theta\in(0,1)$.
\end{lemma}
\begin{proof}
	By H\"older's inequality,
	\begin{align*}
		(1-\theta)\esp[X]
		&=
		\esp\left[X-\theta\esp[X]\right]
		\leq
		\esp\left[(X-\theta\esp[X]){\bf 1}_{X>\theta\esp[X]}\right]
		\\
		&\leq
		\esp\left[(X-\theta\esp[X])^{\nu}{\bf 1}_{X>\theta\esp[X]}\right]^{\frac{1}{\nu}}
		\p\left[
			X>\theta\esp[X]
		\right]^{\frac{\nu-1}{\nu}}
		\\
		&\leq
		\esp[X^{\nu}]^{\frac{1}{\nu}}
		\p\left[
			X>\theta\esp[X]
		\right]^{\frac{\nu-1}{\nu}}.
	\end{align*}
	Hence,
	\begin{align*}
		\p\left[
			X>\theta\esp[X]
		\right]^{\frac{\nu-1}{\nu}}
		&\geq
		(1-\theta)
		\frac{\esp[X]}{\esp[X^{\nu}]^{\frac{1}{\nu}}}
		\geq
		(1-\theta)B^{-\frac{1}{\nu}}. \qedhere
	\end{align*}
\end{proof}



\end{document}